\documentclass{article}
\usepackage{spconf}
\usepackage{amsmath,amsfonts,amssymb,amscd,amsthm}
\usepackage{graphicx,epsfig}
\usepackage{color}
\usepackage{mathtools}
\usepackage{bbm}
\usepackage{url}
\usepackage{balance}
\usepackage{epstopdf}
\usepackage{subfigure}
\usepackage{bm}

\newtheorem{lemma}{\textbf{Lemma}}
\newtheorem{theorem}{\textbf{Theorem}}

\title{Quickest Detection of Growing Dynamic Anomalies in Networks}
\name{Georgios Rovatsos$^{1}$, Venugopal V. Veeravalli$^{1}$, Don Towsley$^{2}$  and Ananthram Swami$^{3}$
\thanks{This work was sponsored in part by the Army Research Laboratory under Cooperative Agreement W911NF-17-2-0196 (IoBT CRA), and in part by the National Science Foundation (NSF) under grant CCF 16-18658.}
}
\address{Department of Electrical and Computer Engineering, University of Illinois at Urbana-Champaign$^{1}$\\
Department of Computer Science, University of Massachusetts Amherst$^{2}$\\
Army Research Lab$^{3}$\\
Email: \small{ \texttt{rovatso2,vvv@illinois.edu, towsley@cs.umass.edu, ananthram.swami.civ@mail.mil}}}

\begin{document}

\maketitle
\begin{abstract}
The problem of quickest growing dynamic anomaly detection in sensor networks is studied. Initially, the observations at the sensors, which are  sampled sequentially by the decision maker, are generated according to a pre-change distribution. At some unknown but deterministic time instant, a dynamic anomaly emerges in the network, affecting different sets of sensors as time progresses. The observations of the affected sensors are generated from a post-change distribution. It is assumed that the number of affected sensors increases with time, and that only the initial and the final size of the anomaly are known to the decision maker. The goal is to detect the emergence of the anomaly as quickly as possible while guaranteeing a sufficiently low frequency of false alarm (FA) events. This detection problem is posed as a stochastic optimization problem by using a delay metric that is based on the worst possible path of the anomaly. A detection rule is proposed that is asymptotically optimal as the mean time to false alarm goes to infinity. Finally, numerical results are provided to validate our theoretical analysis.
\end{abstract}


\section{Introduction}
\label{sec:intro}

Quickest change detection (QCD) has been used to model a wide range of applications in which achieving an accurate real-time estimate of the state of the monitored system is crucial to guarantee its reliability \cite{Rovatsos:2016}--\nocite{Rovatsos2:2016}\nocite{SIM:SIM2032}\nocite{Frisen:2009}\cite{4698342}.
The goal in QCD is to detect a change in the distribution of sequentially observed processes as quickly as possible, subject to false alarm constraints. In the classical QCD problem \cite{veer-bane-elsevierbook-2013,poor-hadj-qcd-book-2009} two standard formulations are used i) the \textit{minimax} setting \cite{lorden:1971}--\nocite{Pollak:1985}\cite{moustakides:1986}, where the changepoint is considered to be deterministic but unknown and the goal is to minimize a worst-case average detection delay subject to a lower bound on the mean time to false alarm; and ii) the \textit{Bayesian} setting \cite{Shiryaev:1963,tart_veera:2005}, where the changepoint is modeled as a random variable with a known distribution, and the goal is to minimize the average detection delay subject to a bound on the probability of false alarm. 

QCD related problems in the context of sensor networks have been extensively studied in the literature \cite{tartakovsky2004change}--\nocite{1677904}\nocite{mei2010efficient}\nocite{7890469}\nocite{xie2013sequential}\nocite{fellouris2016se}\nocite{hadjiliadis2009one}\nocite{raghavan2010quickest}\nocite{Ludkovski2012BayesianQD}\cite{Zou_sens:2018}. Research has been conducted for the case that an unknown subset of sensors switches to the post-change mode, with the distribution change being persistent at these sensors. In \cite{1677904}, an asymptotically optimal procedure was proposed for the case of one sensor being affected persistently after the changepoint. In \cite{mei2010efficient}--\nocite{xie2013sequential}\cite{fellouris2016se}, asymptotically optimal procedures were derived for the case that the anomaly affects an unknown subset of sensor network nodes. In \cite{Zou_sens:2018}, the problem of detecting an anomaly as quickly as possible after it affects more than a pre-defined number of nodes was studied. 

Note that all the works on sensor network detection problems mentioned above have a common element: there is a persistent change in the distribution of each affected sensor after it perceives the anomaly. In this work, we consider the problem of detecting anomalies that may affect different sets of sensors at different time instants, i.e., the sensors may alternate between the pre-change and the post-change mode. We focus on the case of a dynamic anomaly that grows in size, affecting a larger number of sensors as time progresses before achieving its final size. We study this QCD problem under Lorden's minimax framework \cite{lorden:1971}. We assume that the locations of the anomalous nodes are unknown and deterministic, and therefore we modify Lorden's delay metric to consider the worst-path of the anomaly with respect to detection delay. Note that, each interval during which the size of the anomaly is stable, but not equal to its final size, can be considered as a transient phase in the sense of \cite{Zou2017}\nocite{7953065}--\cite{Moustakides:2016}. The main difference between our work and prior studies in the literature on transient QCD is that in our case the locations of the anomalous nodes are not known, and hence the distribution of the observed data is not completely specified after the change. However, by connecting our problem setting with the results in \cite{Zou2017}, we propose a recursive algorithm that we show is asymptotically optimal as the mean time of false alarm (MTFA) goes to infinity.
\vspace{-0.2cm}
\section{Problem Model}
\label{sec:Problem_model}
\vspace{-0.2cm}
Consider a network of $L$ nodes denoted by $[L] \triangleq \{1,\dots,L\}$. Let $\bm{X}[k] = [X_{1}[k],\dots,X_{L}[k]]^\top$ denote the vector comprised of the observations obtained by the network at time $k$, where $X_{\ell}[k]$ denotes the measurement observed by node $\ell \in [L]$ at time $k$. At some \textit{deterministic} but \textit{unknown} time $\nu_1$, a dynamic anomaly appears in the network, initially affecting a set of $m$ nodes, where $m$ is known to the decision maker and the set of affected nodes changes with time. At time $\nu_2 \geq \nu_1$ the anomaly grows to a size of $m +1$ while still being dynamic. This process continues until $\nu_{n -m +1}$ after which the anomaly affects a set of $n$ nodes. As a result, after the pre-change phase we have a sequence of $n -m +1$ phases, where phase $i$ begins at an unknown starting point $\nu_i$ and during this phase an anomaly of size $m+i -1$ moves around the network. Phases $1,\dots,n -m$ can be thought of as transient phases in the sense of \cite{Zou2017}, since the anomaly may grow larger during these phases. Phase $n-m +1$ corresponds to the persistent phase, since in this phase the anomaly settles at its maximum size. 

We assume that the observations are independent across time, conditioned on the values of the changepoints $\nu_i$, $1 \leq i \leq n -m +1$. Furthermore, we assume that the the components of $\bm{X}[k]$ are independent conditioned on the locations of the anomalous nodes at time $k$. For $1 \leq i \leq n -m +1$, define by $\bm{S}^{(i)}[k]$ the location of the anomalous nodes at phase $i$ and time $k$. Then, denote by $\bm{S}^{(i)}=\{\bm{S}^{(i)}[k]\}_{k=1}^\infty$ the unknown but deterministic trajectory of the anomaly at phase $i$. Here, for notational convenience we consider the sequence $\bm{S}^{(i)}$ for all $k\geq 1$, even though the values of the sequence outside phase $i$ do not play a role in the distribution of the observed process. Define by $g(\cdot)$ and $f(\cdot)$ the pre- and post-change pdfs. In this work, we consider the case of \textit{homogeneous} sensors, i.e., the pre- and post-change pdfs are assumed to be the same across sensors. Then, for a fixed set of trajectory sequences $\bm{S} = \{\bm{S}^{(i)}\}_{i=1}^{n -m+1}$ and fixed changepoints $\{\nu_i\}_{i=1}^{ n -m +1}$ we have that for $1\leq k <\nu_1$
\begin{align}
\label{eq:growing_stat_model_1}
\bm{X}[k] \sim g(\bm{X}[k]) \triangleq  \prod\limits_{\ell=1}^L g(X_{\ell}[k])
\end{align}
 and for $\nu_i \leq k < \nu_{i+1} $ (where $\nu_{n - m +2} \triangleq \infty$) we have that
\begin{align}
 \label{eq:growing_stat_model_2}
\bm{X}[k]  \sim  \nonumber
  p_{\bm{S}^{(i)}[k]}(\bm{X}[k])& \triangleq \left( \prod\limits_{\ell \in \bm{S}^{(i)}[k]} f(X_{\ell}[k])\right) \\& \,  \times\left(\prod\limits_{\ell \notin \bm{S}^{(i)}[k]} g(X_{\ell}[k])\right). 
\end{align}
The duration of the $i$-th transient phase is denoted by $d_i \triangleq \nu_{i+1} - \nu_i$ for $ 1 \leq i \leq n -m$. Note that we assume that in addition to the changepoints, the durations of the transient phases are also deterministic but unknown.

Define by $\mathbb{E}_\infty[\cdot]$ the expectation when no anomaly is present. To quantify the frequency of false alarm (FA) events we use the mean time to false alarm (MTFA) denoted by $\mathbb{E}_\infty[\tau]$ for stopping time $\tau$. Furthermore, we use a detection delay metric based on a modification of Lorden's delay \cite{lorden:1971} to account for the fact that the locations of the anomalous nodes are unknown. In particular, define by $\mathcal{F}_k = \sigma(\bm{X}[1],\dots,\bm{X}[k])$ the $\sigma$-algebra generated by $\bm{X}[1],\dots,\bm{X}[k]$. Also, denote by $\mathbb{E}^{\bm{S}}_{\nu,\bm{d}}[\cdot]$ the expectation when $\nu_1 =\nu$, the transient durations are specified by the vector $\bm{d} = [ d_1 ,\dots,d_{n - m}]^\top$, and the trajectory of the anomaly is completely specified by the sequences $\bm{S} = \{\bm{S}^{(i)}\}_{i=1}^{n -m+1}$. To evaluate our detection schemes, we use the following delay metric:
\begin{align}
\label{eq:WADD_growing}
\mathrm{WADD}_{\bm{d}}(\tau) \hspace{-1.1pt} =  \hspace{-0.5pt} \sup\limits_{\bm{S}} \sup_{\nu \geq 1} \mathrm{ess}\hspace{-0.6pt} \sup \hspace{-0.1pt}\mathbb{E}_{\nu,\bm{d}}^{\bm{S}}[\tau - \nu +1 | \mathcal{F}_{\nu -1}, \hspace{-0.2pt}\tau \geq \nu]
\end{align}
where we use the convention that $\mathbb{E}_{\nu,\bm{d}}^{\bm{S}}[\tau - \nu +1 | \mathcal{F}_{\nu -1}, \tau \geq \nu] \triangleq 1$ when $\mathbb{P}_{\nu,\bm{d}}^{\bm{S}}\left(\tau \geq \nu\right)=0$. For $\gamma > 0$, define $\mathcal{C}_\gamma = \{ \tau : \mathbb{E}_\infty[\tau] \geq \gamma\}$. We are interested in solving the following stochastic optimization problem:
\begin{equation}
\label{eq:optimization_2}
\begin{aligned}
& \underset{\tau}{\text{min}}
& & \mathrm{WADD}_{\bm{d}}(\tau) \\
& \text{s.t.}
& & \tau \in \mathcal{C_\gamma}.
\end{aligned}
\end{equation}

Another observation model that will be important for our theoretical analysis is the statistical model which arises when the anomalous nodes at each time instant are chosen uniformly at random from the set of all combinations of size $m +i-1$ in $[L]$, which we denote by $\text{comb}([L],m+i-1)$. According to this model, we have that the observations before $\nu_1$ will be generated according to \eqref{eq:growing_stat_model_1}. After $\nu_1$, the data follow a transient QCD model \cite{Zou2017} where the distribution at each phase is a mixture of distributions, i.e., for $\nu_i \leq k < \nu_{i+1} $ where $1\leq i \leq n - m +1$ we have that
 \begin{align}
 \label{eq:random_stat_model_2}
\bm{X}[k] \sim 
  \overline{p}^{(i)}(\bm{X}[k])\triangleq \sum\limits_{\bm{A} \,\in\, \text{comb}([L],m+i-1)} \frac{p_{\bm{A}}(\bm{X}[k])}{\binom{L}{m+i-1}} .
\end{align}
Then, the Kullback-Liebler (KL) divergence number between the distribution corresponding to phase $i$ and the pre-chage distribution \cite{Zou2017} is given by

\begin{align}
I_i \triangleq \mathbb{E}_{\overline{p}^{(m)}} \left[\log \frac{\overline{p}^{(m)}(\bm{X})}{g(\bm{X})} \right] 
\end{align}
for $1 \leq i \leq n -m +1$. Furthermore, denote by $\overline{\mathbb{E}}_{\nu,\bm{d}}[\cdot]$ the expectation under the statistical model of \eqref{eq:growing_stat_model_1}, \eqref{eq:random_stat_model_2} when the change occurs at time $\nu$ and the transient durations are given by vector $\bm{d}$. Then, the detection delay for the QCD problem characterized by this statistical model is given by
\begin{align}
\label{eq:growing_random_delay}
\overline{\mathrm{WADD}}_{\bm{d}}(\tau) =  \sup_{\nu \geq 1} \mathrm{ess}\hspace{-0.6pt} \sup \overline{\mathbb{E}}_{\nu,\bm{d}}[\tau - \nu +1 | \mathcal{F}_{\nu -1},\vspace{-0.2pt} \tau \geq \nu].
\end{align}

\vspace{-0.2cm}
\section{Mixture-WD-CuSum Algorithm}
\label{sec:algorithms}
\vspace{-0.2cm}

Our proposed detection scheme is based on exploiting the symmetry of the observation model to modify the WD-CuSum algorithm proposed in \cite{Zou2017}. In particular, consider the following test statistic:
\begin{align}
\label{eq:stat_WD}
W[k] = \max\{\Omega^{(1)}[k],\dots,\Omega^{(n -m +1)}[k],0 \},
\end{align}
where for $1\ \leq i \leq n -m +1$, $\Omega^{(i)}$ is calculated as follows:
\begin{align}
\label{eq:stat_WD_2}
\Omega^{(i)}[k] &= \max\limits_{0 \leq j \leq i} \left( \Omega^{(j)}[k-1]+ \sum\limits_{r=j}^{i-1}  \log \rho_r \right) \nonumber \\&+ \log\frac{\overline{p}^{(i)}(\bm{X}[k])}{g(\bm{X}[k])} +\log(1-\rho_i) \
\end{align}
where $\rho_0 \triangleq 1$, $\rho_i \in (0,1)$ for $1 \leq i \leq n -m $, $\rho_{n -m +1}\triangleq 0$, $\Omega^{(i)}[0]\triangleq0$ for all $1 \leq i \leq n -m +1$, and $\Omega^{(0)}[k]\triangleq 0$ for all $k$. Furthermore, define the corresponding stopping time by
\begin{align}
\label{eq:stop_WD}
\tau_{W} = \inf\{ k\geq 1 : W[k] \geq b\}.
\end{align}
Note that to construct the algorithm a mixture approach with respect to $\bm{d}$ with a specific choice of weights that guarantee a recursive test structure is employed. The $\rho_i$ parameters arise from this choice of mixture weights (see \cite{Zou2017} for more details). In Sec. \ref{sec:Main_results}, we will show how to choose the $\rho_i$ parameters to guarantee that our proposed test is asymptotically optimal. Note that our mixture-WD-CuSum scheme is essentially the WD-CuSum algorithm \cite{Zou2017} that detects a transition from the joint distribution in \eqref{eq:growing_stat_model_1} to the mixture of distributions in \eqref{eq:random_stat_model_2}. In the next section we establish that this detection scheme is first-order asymptotically optimal with respect to \eqref{eq:optimization_2}.
\vspace{-0.2cm}
\section{Main Results}
\label{sec:Main_results}
\vspace{-0.2cm}
In this section, we present the asymptotic optimality of the detection scheme presented in \eqref{eq:stat_WD} - \eqref{eq:stop_WD} along with a sketch of the proof. For more details on the analysis see \cite{Rovatsos_sens:2019}. Define by $\Gamma_{\bm{S}}(k,\nu,\bm{d})$ the likelihood ratio at time $k$ between the hypothesis that the dynamic anomaly appears at time $\nu$ and grows with the trajectory of the anomalous nodes being specified by $\bm{S}$, $\bm{d}$, and the hypothesis that the anomaly never appears. Furthermore, define by $\mathcal{L}(k,\nu,\bm{d})$ the likelihood ratio at time $k$ between the hypothesis that the dynamic anomaly appears at time $\nu$ and at each phase $i$ (of duration $d_i$) the anomalous nodes are chosen uniformly at random (see eq. \eqref{eq:random_stat_model_2}), and the hypothesis that the anomaly never appears.

To ensure that the transient phases play a non-trivial role asymptotically, we let the durations of the transient periods go to infinity with $\gamma$. In particular, without loss of generality, we assume that there exist constants $c_i \in [0,\infty) \cup \{\infty\}, 1\leq i \leq n -m$ and $c_{n - m+1} \triangleq \infty$ such that as $\gamma \rightarrow \infty$

\begin{align}
\label{eq:d_scaling_general}
d_i \sim c_i \frac{\log\gamma}{I_{m+i-1}}.
\end{align}
This assumption can be intuitively explained since asymptotically the rate of the transient durations with respect to $\log \gamma$ will indicate the phase at which the anomaly will be detected (\cite{Zou2017}). To establish the asymptotic optimality of our proposed detection procedure, we connect the worst-path QCD problem defined in eqs. \eqref{eq:growing_stat_model_1} - \eqref{eq:optimization_2} with the transient QCD problem defined in \eqref{eq:growing_stat_model_1}, \eqref{eq:random_stat_model_2}, \eqref{eq:growing_random_delay} and use the theoretical results of \cite{Zou2017}. We start by proving a lemma that connects $\mathrm{WADD}_{\bm{d}}$ with $\overline{\mathrm{WADD}}_{\bm{d}}$ for an arbitrary stopping time $\tau$.
\begin{lemma}
\label{thm:lemma_1}
For any stopping rule $\tau$ and any $\bm{d}$ we have that
\begin{align}
\label{eq:lemma_1_eq}
\mathrm{WADD}_{\bm{d}}(\tau) \geq \overline{\mathrm{WADD}}_{\bm{d}}(\tau).
\end{align}
\end{lemma}
\begin{proof}[Proof sketch]
The proof begins by establishing that for any $\tau$ and $\tau^{(N)} \triangleq \min\{\tau,N\}$, where $N >0$, we have that
\begin{align}
\mathrm{WADD}_{\bm{d}}(\tau)  \geq \mathrm{WADD}_{\bm{d}}(\tau^{(N)}).
\end{align}
Then, by a change of measure argument we show that 
\begin{align}
&\mathrm{WADD}_{\bm{d}}(\tau^{(N)}) \geq \sup_{\bm{S}} \mathbb{E}_{\nu,\bm{d}}^{\bm{S}}[\tau^{(N)} - \nu +1 | \mathcal{F}_{\nu -1}, \tau^{(N)} \geq \nu] \nonumber\\& = \mathbb{E}_{\infty}\left[  \mathbbm{1}_{\{\tau^{(N)} \geq \nu\}}  \bigg| \mathcal{F}_{\nu -1}, \tau^{(N)} \geq \nu\right]\nonumber \\&   + \sup_{\bm{S}}\mathbb{E}_{\infty}\left[ \sum\limits_{i =\nu}^{N -1}  \Gamma_{\bm{S}} \left(i,\nu,\bm{d}\right)   \mathbbm{1}_{\{i < \tau^{(N)} \}}  \bigg| \mathcal{F}_{\nu -1}, \tau^{(N)} \geq \nu\right],
\end{align}
for any $\nu$, $\bm{d}$. Then, by using the fact that the supremum of a set of numbers is lower bounded by the average, and by using another change of measure argument we can establish that
\begin{align}
\label{eq:ineq_1}
&\mathrm{WADD}_{\bm{d}}(\tau^{(N)})  \geq \mathbb{E}_{\infty}\left[  \mathbbm{1}_{\{\tau^{(N)} \geq \nu\}}  \bigg| \mathcal{F}_{\nu -1}, \tau^{(N)} \geq \nu\right]\nonumber \\&   + \mathbb{E}_{\infty}\left[ \sum\limits_{i =\nu}^{N -1}  \mathcal{L} \left(i,\nu,\bm{d}\right)   \mathbbm{1}_{\{i < \tau^{(N)} \}}  \bigg| \mathcal{F}_{\nu -1}, \tau^{(N)} \geq \nu\right] \nonumber \\&  =\overline{\mathbb{E}}_{\nu,\bm{d}}[\tau^{(N)} - \nu +1 | \mathcal{F}_{\nu -1}, \tau^{(N)} \geq \nu].
\end{align}
The lemma is then established from the Monotone Convergence Theorem by taking the limit as $N \rightarrow \infty$, and using the fact that inequality \eqref{eq:ineq_1} holds for all $\nu$, $\bm{d}$.
\end{proof}
Note that, by using Lemma \ref{thm:lemma_1} and the universal asymptotic lower bound for transient QCD \cite{Zou2017} we can derive a universal asymptotic lower bound on the worst-path delay $\mathrm{WADD}_{\bm{d}}$. In view of Lemma \ref{thm:lemma_1}, and since the stopping rule presented in \eqref{eq:stat_WD} - \eqref{eq:stop_WD} solves the transient QCD problem specified by eqs. \eqref{eq:growing_stat_model_1}, \eqref{eq:random_stat_model_2}, \eqref{eq:growing_random_delay} asymptotically, the asymptotic optimality of the mixture-WD-CuSum test will be established if we can show that equality in \eqref{eq:lemma_1_eq} is attained when $\tau = \tau_W$. In particular, we have the following lemma:
\begin{lemma}
\label{thm:lemma_2}
For the stopping rule defined in \eqref{eq:stat_WD} - \eqref{eq:stop_WD} and for any $\bm{d}$ we have that
\begin{align}
\mathrm{WADD}_{\bm{d}}(\tau_W) = \overline{\mathrm{WADD}}_{\bm{d}}(\tau_W).
\end{align}
\end{lemma}
\begin{proof}[Proof sketch]
Since the worst-case delay of the test in \eqref{eq:stat_WD} - \eqref{eq:stop_WD} is attained at $\nu=1$, by performing a change of measure we can show that
\begin{align}
\label{eq:yo_1}
\mathrm{WADD}_{\bm{d}}(\tau_W) = \lim\limits_{N \rightarrow \infty} \sum\limits_{i=1}^N i \mathbb{E}_{\infty}\left[\Gamma_{\bm{S}}(i,1,\bm{d})\mathbbm{1}_{\{\tau_W = i\}} \right].
\end{align}
Then, due to the symmetry of the proposed test and by bounding the sup by the average as in Lemma 1 we can show that
\begin{align}
\label{eq:yo_2}
 &\sum\limits_{i=1}^N i \mathbb{E}_{\infty}\left[\left(\Gamma_{\bm{S}}(i,1,\bm{d})-\mathcal{L}(i,1,\bm{d})\right)\mathbbm{1}_{\{\tau_W = i\}} \right] = 0.
\end{align}
Finally, by using eqs. \eqref{eq:yo_1}, \eqref{eq:yo_2} and a change of measure argument the lemma is established.
\end{proof}
Combining Lemmas 1, 2 and the theoretical results of \cite{Zou2017} we derive the asymptotic optimality of our test.
\begin{theorem}
Consider the QCD problem described in Sec. \ref{sec:Problem_model}. We have the following:

i) $b = \log \gamma$ implies that
\begin{align} 
\mathbb{E}_\infty[\tau_{W}] \geq \gamma.
\end{align}

ii) If $b=\log \gamma$, $\bm{d} $ satisfies \eqref{eq:d_scaling_general} as $\gamma \rightarrow \infty$ for some constants $c_i \in [0,\infty) \cup \{\infty\}, 1 \leq i\leq n -m$ and $c_{n - m+1} \triangleq \infty$, and $\rho_i$ is chosen such that
\begin{align}
\label{eq:important_condition}
\rho_i \rightarrow 0 \text{\quad and \quad} \frac{-\log\rho_i}{\log\gamma} \rightarrow 0,
\end{align}
for $1 \leq i\leq n -m$ as $\gamma \rightarrow \infty$ we have that
\begin{align}
&\inf_{\tau \in \mathcal{C}_\gamma } \mathrm{WADD}_{\bm{d}}(\tau)  \sim  \mathrm{WADD}_{\bm{d}}(\tau_{W}) \nonumber \\& \sim \log\gamma \left(\sum_{i=1}^{h-1}\frac{c_i}{I_{m+i-1}} + \frac{1 - \sum_{i=1}^{h-1} c_i }{I_{m+h-1}} \right),
\end{align}
where $h = \min\{1 \leq j \leq n -m +1 : \sum_{i=1}^j c_i\geq 1\}$.
\end{theorem}

\vspace{-0.2cm}
\section{Simulation Results and Discussion}
\label{sec:Numerical}
\vspace{-0.2cm}
In this section, we numerically evaluate the performance of the algorithm proposed in eqs. \eqref{eq:stat_WD} - \eqref{eq:stop_WD} for the case of $g = \mathcal{N}(0,1)$, $f = \mathcal{N}(1,1)$ and different $L$. Note that $\mathrm{WADD}_{\bm{d}}$ for the mixture-WD-CuSum test is attained at $\nu_1 = 1$, which together with the symmetry of our scheme implies that $\mathrm{WADD}_{\bm{d}}$ can be numerically evaluated. We use $\rho_i = \frac{1}{b}$ to guarantee that the conditions in \eqref{eq:important_condition} are satisfied.

In Fig. 1, we simulate mixture-WD-CuSum for the case of $m=1$, $n=3$, $d_1 = 9$, $d_2 =10$ and for $L=3,5,10$. Note that, as expected, for a fixed MTFA value the detection delay increases as the network size increases. This is justified since a larger network size implies that in \eqref{eq:stat_WD_2} the impact of the terms in the summation in $\overline{p}^{(i)}(\cdot)$ that will not contribute to the detection is more significant. Furthermore, we see that as the MTFA increases the slopes of the curves decrease gradually. This means that the mixture-WD-CuSum is adaptive to each transient phase (\cite{Zou2017}), since $I_i > I_{i'}$ for $i<i'$.

In Fig. 2, we inspect how much loss we suffer due to lack of knowledge of $m$ and $n$ by the decision maker. In particular, we consider the case of $m =2$, $n=4$, $d_1 = 9$, $d_2 =10$ and $L=6$ and compare the performance of the mixture-WD-CuSum test that exploits the knowledge of $m$ and $n$ with the mixture-WD-CuSum that assumes that $m=1$, $n=6$, due to the lack of knowledge of their true values. Note that, as expected, the algorithm that exploits the knowledge of $m$ and $n$ offers superior performance. It should be noted that the performance loss for our case study is not significant. Furthermore, note that partial information regarding $m$ or $n$ (e.g. bounds on their values) may be available in practical applications, which can be used to reduce the performance gap between the two procedures. However, it is expected that the performance loss will be significant as $L$ increases, if our estimates for $m$ and $n$ are not sufficiently accurate. 

\vspace{-0.2cm}
\begin{figure}[!htp]
\label{fig:fig_1}
\hspace{-0.51cm}
\includegraphics[width=9.8cm]{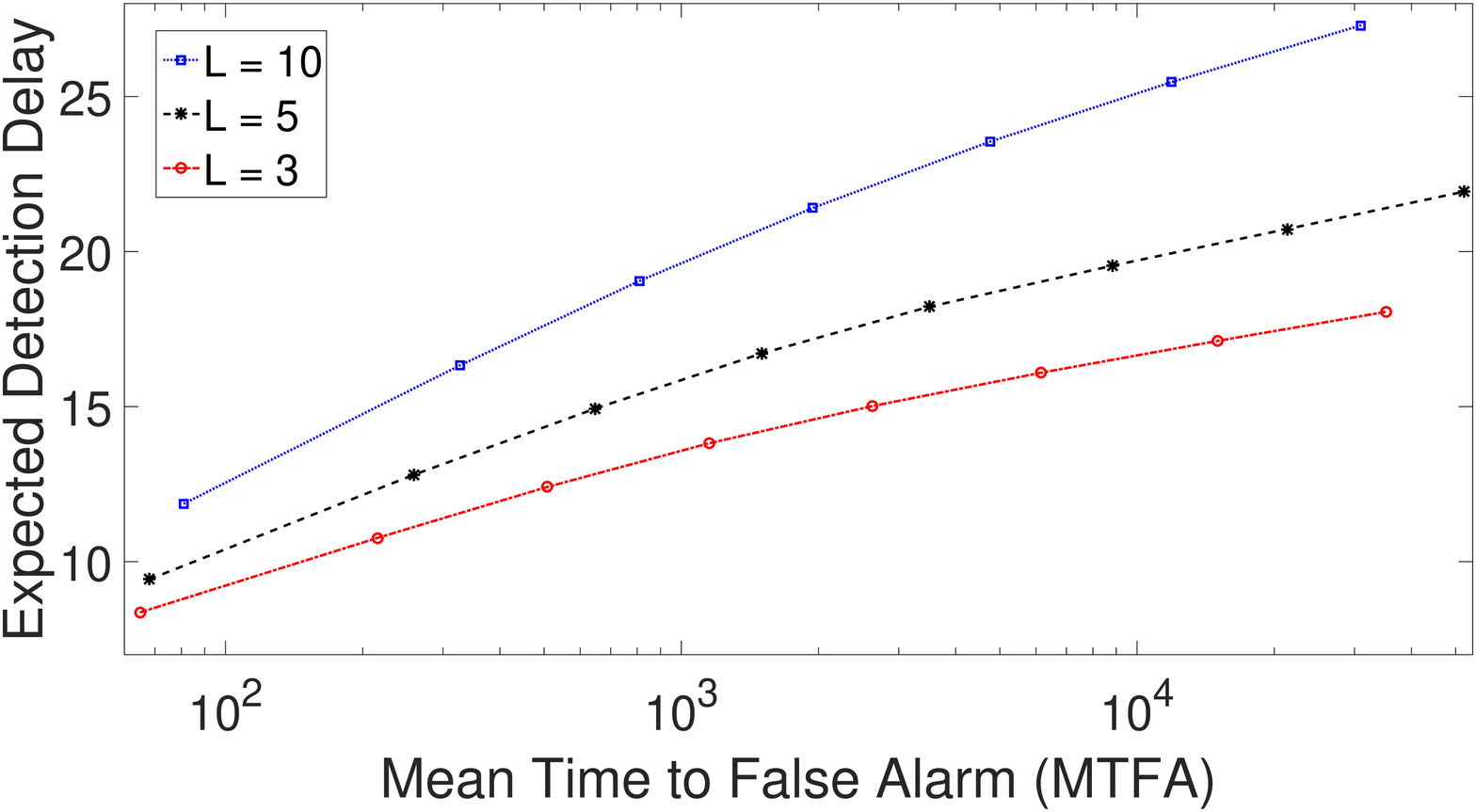}
\caption{$\mathrm{WADD}_{\bm{d}}$ vs MTFA for $m =1$, $n=3$ and varying network sizes.}
\end{figure}

\vspace{-0.2cm}
\begin{figure}[!htp]
\label{fig:fig_2}
\hspace{-0.51cm}
\includegraphics[width=9.8cm]{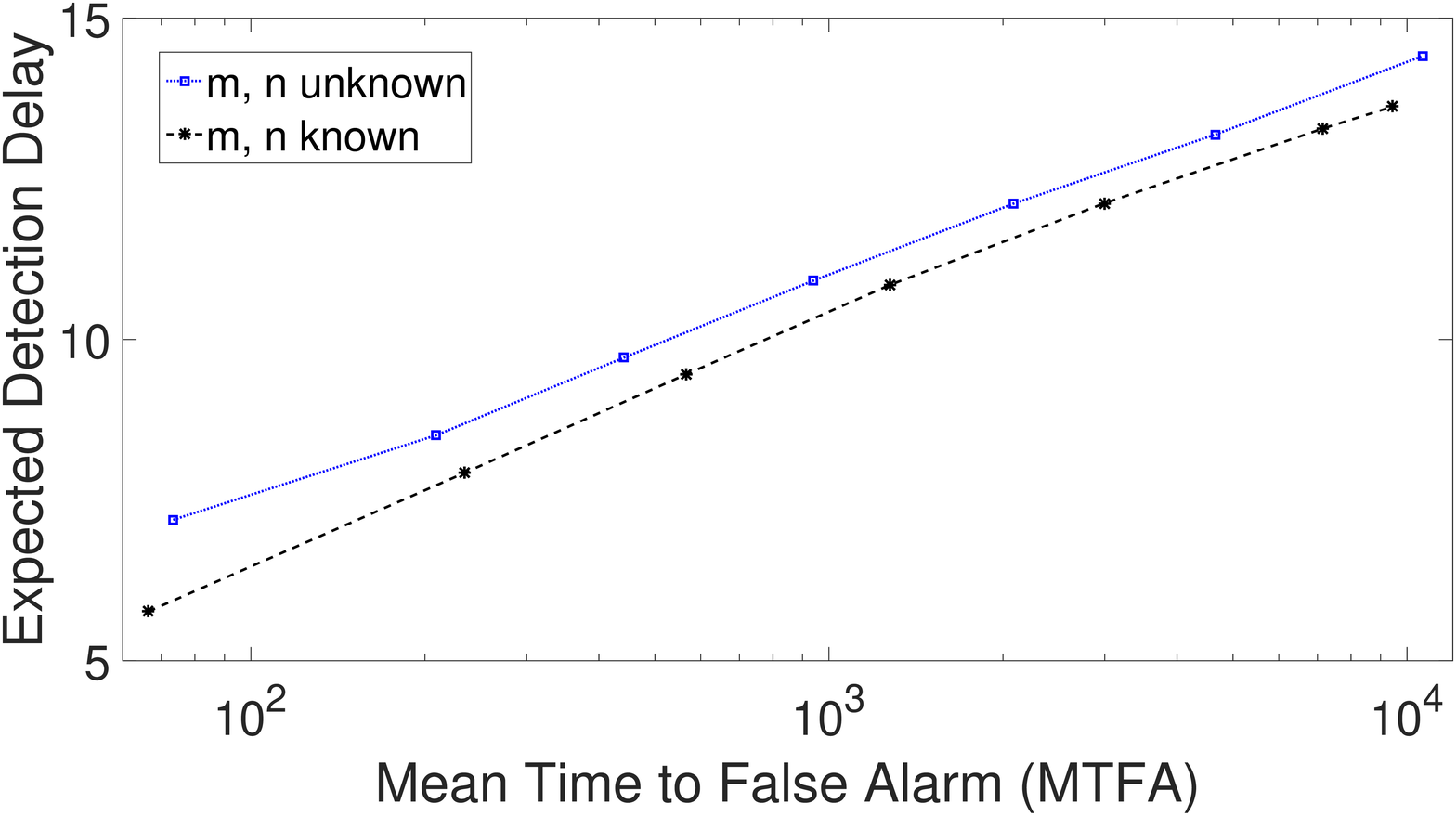}
\caption{$\mathrm{WADD}_{\bm{d}}$ vs MTFA comparison between the test that exploits and test that does not exploit knowledge of $m$ and $n$, for $m =2$, $n=4$ and $L=6$ .}\label{fig:1}
\end{figure}

\newpage

\bibliographystyle{IEEEbib}
\balance
\bibliography{CC_bibliography}

\begin{thebibliography}{10}

\bibitem{Rovatsos:2016}
G.~Rovatsos, X.~Jiang, A.~D. Dom\'inguez-Garc\'ia, and V.~V. Veeravalli,
\newblock ``Comparison of statistical algorithms for power system line outage
  detection,''
\newblock in {\em Proc. IEEE Int. Conf. Acoust. Speech Signal Process.
  (ICASSP)}, Shanghai, China, Mar. 2016, pp. 2946--2950.

\bibitem{Rovatsos2:2016}
G.~Rovatsos, J.~Jiang, A.~D. Dom\'{i}nguez-Garc\'{i}a, and V.~V. Veeravalli,
\newblock ``Statistical power system line outage detection under transient
  dynamics,''
\newblock {\em IEEE Trans. Signal Proc.}, vol. 65, no. 11, pp. 2787--2797, Jun.
  2017.

\bibitem{SIM:SIM2032}
S.~E. Fienberg and G.~Shmueli,
\newblock ``Statistical issues and challenges associated with rapid detection
  of bio-terrorist attacks,''
\newblock {\em Statistics in Medicine}, vol. 24, no. 4, pp. 513--529, Feb.
  2005.

\bibitem{Frisen:2009}
M.~Frisn,
\newblock ``Optimal sequential surveillance for finance, public health, and
  other areas,''
\newblock {\em Sequential Analysis}, vol. 28, no. 3, pp. 310--337, Jul. 2009.

\bibitem{4698342}
L.~Lai, Y.~Fan, and H.~V. Poor,
\newblock ``Quickest detection in cognitive radio: A sequential change
  detection framework,''
\newblock in {\em Proc. Glob. Tel. Conf. (GLOBECOM)}, Nov. 2008, pp. 1--5.

\bibitem{veer-bane-elsevierbook-2013}
V.~V. Veeravalli and T.~Banerjee,
\newblock {\em Quickest Change Detection},
\newblock Elsevier: E-reference Signal Processing, 2013.

\bibitem{poor-hadj-qcd-book-2009}
H.~V. Poor and O.~Hadjiliadis,
\newblock {\em Quickest Detection},
\newblock Cambridge University Press, 2009.

\bibitem{lorden:1971}
G.~Lorden,
\newblock ``Procedures for reacting to a change in distribution,''
\newblock {\em Annals of Mathematical Statistics}, vol. 42, no. 6, pp.
  1897--1908, Dec. 1971.

\bibitem{Pollak:1985}
M.~Pollak,
\newblock ``Optimal detection of a change in distribution,''
\newblock {\em Annals of Statistics}, vol. 13, no. 1, pp. 206--227, Mar. 1985.

\bibitem{moustakides:1986}
G.~V. Moustakides,
\newblock ``Optimal stopping times for detecting changes in distributions,''
\newblock {\em Annals of Statistics}, vol. 14, no. 4, pp. 1379--1387, Dec.
  1986.

\bibitem{Shiryaev:1963}
A.~N. Shiryaev,
\newblock ``On optimum methods in quickest detection problems,''
\newblock {\em Theory of Probability \& Its Applications}, vol. 8, no. 1, pp.
  22--46, 1963.

\bibitem{tart_veera:2005}
A.~G. Tartakovsky and V.~V. Veeravalli,
\newblock ``General asymptotic {B}ayesian theory of quickest change
  detection,''
\newblock {\em Theory of Probability \& Its Applications}, vol. 49, Jan. 2005.

\bibitem{tartakovsky2004change}
A.~G. Tartakovsky and V.~V. Veeravalli,
\newblock ``Change-point detection in multichannel and distributed systems,''
\newblock {\em Applied Sequential Methodologies: Real-World Examples with Data
  Analysis}, vol. 173, pp. 339--370, Jan. 2004.

\bibitem{1677904}
A.~G. Tartakovsky, B.~L. Rozovskii, R.~B. Blazek, and H.~Kim,
\newblock ``A novel approach to detection of intrusions in computer networks
  via adaptive sequential and batch-sequential change-point detection
  methods,''
\newblock {\em IEEE Trans. Signal Proc.}, vol. 54, no. 9, pp. 3372--3382, Sept.
  2006.

\bibitem{mei2010efficient}
Y.~Mei,
\newblock ``Efficient scalable schemes for monitoring a large number of data
  streams,''
\newblock {\em Biometrika}, vol. 97, no. 2, pp. 419--433, Apr. 2010.

\bibitem{7890469}
G.~Fellouris and A.~G. Tartakovsky,
\newblock ``Multichannel sequential detection --- part {I}: Non-i.i.d. data,''
\newblock {\em IEEE Trans. Inform. Theory}, vol. 63, no. 7, pp. 4551--4571,
  Jul. 2017.

\bibitem{xie2013sequential}
Y.~Xie and D.~Siegmund,
\newblock ``Sequential multi-sensor change-point detection,''
\newblock {\em Annals of Statistics}, vol. 41, no. 2, pp. 670--692, 2013.

\bibitem{fellouris2016se}
G.~Fellouris and G.~Sokolov,
\newblock ``Second-order asymptotic optimality in multisensor sequential change
  detection,''
\newblock {\em IEEE Trans. Inform. Theory}, vol. 62, no. 6, pp. 3662--3675,
  2016.

\bibitem{hadjiliadis2009one}
O.~Hadjiliadis, H.~Zhang, and H.~V. Poor,
\newblock ``One shot schemes for decentralized quickest change detection,''
\newblock {\em IEEE Trans. Inform. Theory}, vol. 55, no. 7, pp. 3346--3359,
  2009.

\bibitem{raghavan2010quickest}
V.~Raghavan and V.~V. Veeravalli,
\newblock ``Quickest change detection of a {M}arkov process across a sensor
  array,''
\newblock {\em IEEE Trans. Inform. Theory}, vol. 56, no. 4, pp. 1961--1981,
  2010.

\bibitem{Ludkovski2012BayesianQD}
M.~Ludkovski,
\newblock ``Bayesian quickest detection in sensor arrays,''
\newblock {\em Sequential Analysis}, vol. 31, no. 4, pp. 481--504, 2012.

\bibitem{Zou_sens:2018}
S.~{Zou}, V.~V. {Veeravalli}, J.~{Li}, and D.~{Towsley},
\newblock ``{Quickest detection of dynamic events in networks},''
\newblock {\em arXiv e-prints}, p. arXiv:1807.06143, Jul. 2018.

\bibitem{Zou2017}
S.~Zou, G.~Fellouris, and V.~V. Veeravalli,
\newblock ``Quickest change detection under transient dynamics: {T}heory and
  asymptotic analysis,''
\newblock {\em IEEE Trans. Inform. Theory}, 2018.

\bibitem{7953065}
G.~Rovatsos, S.~Zou, and V.~V. Veeravalli,
\newblock ``Quickest change detection under transient dynamics,''
\newblock in {\em Proc. IEEE Int. Conf. Acoust. Speech Signal Process.
  (ICASSP)}, Mar. 2017.

\bibitem{Moustakides:2016}
G.~V. Moustakides and V.~V. Veeravalli,
\newblock ``Sequentially detecting transitory changes,''
\newblock in {\em Proc. IEEE Int. Symp. Inform. Theory (ISIT)}, Jul. 2016.

\bibitem{Rovatsos_sens:2019}
G.~{Rovatsos}, V.~V. {Veeravalli}, D.~{Towsley}, and A.~{Swami},
\newblock ``{Quickest detection of growing dynamic anomalies in networks},''
\newblock {\em arXiv e-prints}, p. arXiv:1910.09151, Oct. 2019.

\end{thebibliography}


\newpage
\onecolumn

\section{Appendix}

For our theoretical analysis, we focus on the case that after the first changepoint (which in this section will be denoted by $\nu$) a single sensor is affected initially ($m =1$) with the anomaly growing to a persistent size of $n = 2$ affected sensors while still being dynamic in nature. The results in this paper hold for the case of arbitrary $1\leq m \leq n \leq L$ known by the decision maker, but in that case the analysis becomes cumbersome. 

Consider the sequences $S^{(1)} = \{{S^{(1)}[k]}\}_{k=1}^{\infty}$ and $\bm{S}^{(2)} = \{{\bm{S}^{(2)}[k]}\}_{k=1}^{\infty}$ where $ \bm{S}^{(2)}[k] = \left[ S_1^{(2)}[k] , S^{(2)}_2[k] \right]^\top$ with $S^{(1)}[k], S^{(2)}_1[k],S^{(2)}_2[k] \in [L] $, which characterize the location of the anomalous nodes at each time instant. Then, for fixed changepoints $\nu$ and $\nu+d$, $\nu \geq 1$, $d\geq 0$, we have the following statistical model:
\small
\begin{align}
\label{eq:stat_model}
\bm{X}[k] \sim \left\{
\begin{array}{ll}
      g\left(\bm{X}[k]\right)  \triangleq \prod\limits_{\ell =1}^L g\left(X_\ell[k]\right) & 1\leq k \leq \nu -1 \\
      p_{S^{(1)}[k]}\left(\bm{X}[k]\right) \triangleq f\left(X_{S^{(1)}[k]}[k]\right) \prod\limits_{\ell \neq S^{(1)}[k]} g\left(X_{\ell}[k]\right) & \nu \leq k \leq \nu +d -1 \\
      p_{\bm{S}^{(2)}[k]}\left(\bm{X}[k]\right) \triangleq f\left(X_{S^{(2)}_1[k]}[k]\right) f\left(X_{S^{(2)}_2[k]}[k]\right) \prod\limits_{\ell \neq S^{(2)}_1[k],S^{(2)}_2[k]} g\left(X_{\ell}[k]\right) & k \geq \nu +d. \\
\end{array} 
\right. 
\end{align}
\normalsize
For fixed set of sequences $\bm{S} = \{ \bm{S}^{(1)} , \bm{S}^{(2)}\}$, and constants $\nu,d$ 
define by $\Gamma_{\bm{S}}(k,\nu,d)$ the likelihood ratio at time $k$ between the hypothesis that the dynamic anomaly appears at time $\nu$ and grows to a size of $2$ at time $\nu+d$ with the trajectory of the anomalous nodes being specified by $\bm{S}$, and the hypothesis that the anomaly never appears. As a result, we have that 

\begin{align}
\Gamma_{\bm{S}}(k,\nu,d) = \left[\prod\limits_{j=\nu}^{\min\{ k , \nu+d-1\}}  \frac{f(X_{S^{(1)}[j]}[j])}{g(X_{S^{(1)}[j]}[j])}  \right]\cdot \left[\prod\limits_{i=\nu+d}^{k}  \frac{f(X_{S^{(2)}_1[j]}[i])f(X_{S^{(2)}_2[j]}[j])}{g(X_{S^{(2)}_1[j]}[j])g(X_{S^{(2)}_2[j]}[j])}  \right].
\end{align}
Furthermore, define by $\mathcal{L}(k,\nu,d)$ the likelihood ratio at time $k$ between the hypothesis that the dynamic anomaly appears at time $\nu$ and grows to a size of $2$ at time $\nu+d$ with the trajectory of the anomalous nodes chosen randomly as in \eqref{eq:growing_stat_model_1}, \eqref{eq:random_stat_model_2} , and the hypothesis that the anomaly never appears. Then, we have that 
\begin{align}
\mathcal{L} (k,\nu,d) & \nonumber=\left[ \prod\limits_{i=\nu}^{\min\{k,\nu+d-1\}} \frac{\overline{p}^{(1)}(\bm{X}[i])}{g(\bm{X}[i])} \right]\cdot \left[ \prod\limits_{i=\nu+d}^{k} \frac{\overline{p}^{(2)}(\bm{X}[i])}{g(\bm{X}[i])} \right] 
\\&
 =  \left[ \prod\limits_{i=\nu}^{\min\{k,\nu+d-1\}}\left(\sum\limits_{\ell=1}^L \frac{1}{L} \frac{  f(X_{\ell}[i])}{  g(X_{\ell}[i])} \right)\right]\cdot \left[ \prod\limits_{i=\nu+d}^{k} \left(  \sum\limits_{(\ell, \ell') \in \, \text{comb}(L,2)} \, \frac{1}{\binom L2} \frac{f(X_{\ell}[i]) f(X_{\ell'}[i])}{g(X_{\ell}[i]) g(X_{\ell'}[i]} \right)\right].
\end{align}

Note that for any sequence $\{\alpha[k]\}_{k=1}^\infty$ we use the convention that $\sum\limits_{k_1}^{k_2}\alpha[k] \triangleq 0$ and $\prod\limits_{k_1}^{k_2} \alpha[k]\triangleq 1$ when $k_1 > k_2$ which implies that $\Gamma_{\bm{S}}(k_2,k_1,d) \triangleq 1$ and $\mathcal{L}(k_2,k_1,d) \triangleq 1$ when $k_1 > k_2$. Furthermore, we define by $a[k_2,k_1]$ the vector consisting of the terms $\left[a[k_2],\dots,a[k_1]\right]^\top.$

Before establishing Lemmas \ref{thm:lemma_1} and \ref{thm:lemma_2} we prove a theorem connecting the delay of a stopping rule $\tau$ to the delay of its truncated version.
\begin{lemma}
\label{thm:lemma_trunc}
For any stopping rule $\tau$, define its truncated version by $\tau^{(N)} \triangleq \min \{ \tau, N\}$ where $N$ is a positive integer. Then, we have that for any $d \geq 0$
\begin{align}
\label{eq:math_4}
\mathrm{WADD}_d(\tau)  \geq\mathrm{WADD}_d(\tau^{(N)}) .
\end{align}
\begin{proof}
Fix $N\geq 1$. Consider initially the case that $N\geq\nu$. Then, since $\{\tau^{(N)} \geq\nu\} = \{\min\{\tau,N\} \geq\nu\} = \{\tau\geq\nu\} \cap \{N\geq\nu\}$, we have that $\{\tau^{(N)}\geq\nu\} = \{\tau\geq\nu\}.$ Since $\tau^{(N)} \leq \tau$, this implies that for any $N \geq \nu$ and any $\bm{S}$ we have that
\begin{align}
\label{eq:math_1}
\mathbb{E}_{\nu,d}^{\bm{S}}[\tau^{(N)} - \nu +1 | \mathcal{F}_{\nu -1}, \tau^{(N)} \geq \nu] =\mathbb{E}_{\nu,d}^{\bm{S}}[\tau^{(N)} - \nu +1 | \mathcal{F}_{\nu -1}, \tau \geq \nu] \leq \mathbb{E}_{\nu,d}^{\bm{S}}[\tau - \nu +1 | \mathcal{F}_{\nu -1}, \tau \geq \nu].
\end{align}
For the case that $N<\nu$, we have that that $\mathbb{P}_{\nu,d}^{\bm{S}}(\tau^{(N)} \geq \nu)=0$, which implies that for any $N < \nu$ and any $\bm{S}$ we have that
\begin{align}
\label{eq:math_2}
\mathbb{E}_{\nu,d}^{\bm{S}} \left[\tau^{(N)}-\nu+1 |\mathcal{F}_{\nu-1},\tau^{(N)}\geq \nu  \right]=1,
\end{align}
by convention. Furthermore, note that for any $\bm{S}$ we have that
\begin{align}
\label{eq:math_3}
\mathbb{E}_{\nu,d}^{\bm{S}} \left[\tau-\nu+1 |\mathcal{F}_{\nu-1},\tau\geq \nu  \right]\geq 1.
\end{align}
From \eqref{eq:math_1} - \eqref{eq:math_3} we have that for any $\nu \geq 1$ and any $\bm{S}$
\begin{align}
\mathbb{E}_{\nu,d}^{\bm{S}}[\tau^{(N)} - \nu +1 | \mathcal{F}_{\nu -1}, \tau^{(N)} \geq \nu] \leq \mathbb{E}_{\nu,d}^{\bm{S}}[\tau - \nu +1 | \mathcal{F}_{\nu -1}, \tau \geq \nu].
\end{align}
By taking the sup and $\text{ess}\sup$ on both sides the lemma is established.
\end{proof}
\end{lemma}
Now we are ready to establish Lemma \ref{thm:lemma_1}.

\begin{proof}[Proof of Lemma 1]
From Lemma \ref{thm:lemma_trunc}, we have that for any $\nu$, $d$, $N\geq 1$ 
\begin{align}
\label{eq:eqe}
\mathrm{WADD}_d(\tau) \geq \mathrm{WADD}_d(\tau^{(N)}) \geq \sup_{\bm{S}}\mathbb{E}_{\nu,d}^{\bm{S}}[\tau^{(N)} - \nu +1 | \mathcal{F}_{\nu -1}, \tau^{(N)} \geq \nu].
\end{align}
Following, we have that for any $\nu, d, \bm{S}$ and $N > \nu+d$
\begin{align}
\label{eq:eq_similar}
&\nonumber \mathbb{E}_{\nu,d}^{\bm{S}}[\tau^{(N)} - \nu +1 | \mathcal{F}_{\nu -1}, \tau^{(N)} \geq \nu] = \mathbb{E}_{\nu,d}^{\bm{S}}\left[\sum\limits_{i =\nu}^\infty \mathbbm{1}_{\{\tau^{(N)} \geq i\}} \bigg| \mathcal{F}_{\nu -1}, \tau^{(N)} \geq \nu\right] \\& \stackrel{(\text{a})}{=}  \mathbb{E}_{\nu,d}^{\bm{S}}\left[\sum\limits_{i =\nu}^N \nonumber\mathbbm{1}_{\{\tau^{(N)} \geq i\}} \bigg| \mathcal{F}_{\nu -1}, \tau^{(N)} \geq \nu\right] =  \sum\limits_{i =\nu}^N \mathbb{E}_{\nu,d}^{\bm{S}}\left[\mathbbm{1}_{\{\tau^{(N)} \geq i\}} \bigg| \mathcal{F}_{\nu -1}, \tau^{(N)} \geq \nu\right] \\&\stackrel{(\text{b})}{=}   \sum\limits_{i =\nu}^N \mathbb{E}_{\infty}\left[\mathbbm{1}_{\{\tau^{(N)} \geq i\}} \Gamma_{\bm{S}}\left(i-1,\nu,d\right) \bigg| \mathcal{F}_{\nu -1}, \tau^{(N)} \geq \nu\right] =  \mathbb{E}_{\infty}\left[ \sum\limits_{i =\nu}^N \mathbbm{1}_{\{\tau^{(N)} \geq i\}} \Gamma_{\bm{S}}\left(i-1,\nu,d\right) \bigg| \mathcal{F}_{\nu -1}, \tau^{(N)} \geq \nu\right] \nonumber \\&=  \mathbb{E}_{\infty}\left[  \mathbbm{1}_{\{\tau^{(N)} \geq \nu\}} \Gamma_{\bm{S}}\left(\nu-1,\nu,d\right) \bigg| \mathcal{F}_{\nu -1}, \tau^{(N)} \geq \nu\right]  +\mathbb{E}_{\infty}\left[ \sum\limits_{i =\nu+1}^N \mathbbm{1}_{\{\tau^{(N)} \geq i\}} \Gamma_{\bm{S}}\left(i-1,\nu,d\right) \bigg| \mathcal{F}_{\nu -1}, \tau^{(N)} \geq \nu\right] \nonumber \\&=  \mathbb{E}_{\infty}\left[  \mathbbm{1}_{\{\tau^{(N)} \geq \nu\}} \nonumber \bigg| \mathcal{F}_{\nu -1}, \tau^{(N)} \geq \nu\right]  +\mathbb{E}_{\infty}\left[ \sum\limits_{i =\nu+1}^N \mathbbm{1}_{\{\tau^{(N)} \geq i\}} \Gamma_{\bm{S}}\left(i-1,\nu,d\right) \bigg| \mathcal{F}_{\nu -1}, \tau^{(N)} \geq \nu\right] \\& \stackrel{(\text{c})}{=} \mathbb{E}_{\infty}\left[  \mathbbm{1}_{\{\tau^{(N)} \geq \nu\}} \bigg| \mathcal{F}_{\nu -1}, \tau^{(N)} \geq \nu\right]  +\mathbb{E}_{\infty}\left[ \sum\limits_{i =\nu}^{N-1} \mathbbm{1}_{\{\tau^{(N)} > i\}} \Gamma_{\bm{S}}\left(i,\nu,d\right) \bigg| \mathcal{F}_{\nu -1}, \tau^{(N)} \geq \nu\right]   ,
\end{align}
where (a) follows since $\mathbbm{1}_{\{\tau^{(N)} >i\}}= 0$ for $i>N$ because $\tau^{(N)} \leq N$, (b) follows from a change of measure, and (c) from a change of variables. As a result, by taking the supremum over $\bm{S}$ we have that
\small
\begin{align}
\label{eq:eqf}
&\sup_{\bm{S}} \mathbb{E}_{\nu,d}^{\bm{S}}[\tau^{(N)} - \nu +1 | \mathcal{F}_{\nu -1}, \tau^{(N)} \geq \nu] = \mathbb{E}_{\infty}\left[  \mathbbm{1}_{\{\tau^{(N)} \geq \nu\}} \nonumber \bigg| \mathcal{F}_{\nu -1}, \tau^{(N)} \geq \nu\right] + \nonumber\sup_{\bm{S}}\mathbb{E}_{\infty}\left[ \sum\limits_{i =\nu}^{N -1}  \Gamma_{\bm{S}} \left(i,\nu,d\right)  \mathbbm{1}_{\{i < \tau^{(N)} \}}  \bigg| \mathcal{F}_{\nu -1}, \tau^{(N)} \geq \nu\right] \\&\stackrel{(\text{e})}{=} \mathbb{E}_{\infty}\left[  \mathbbm{1}_{\{\tau^{(N)} \geq \nu\}} \nonumber \bigg| \mathcal{F}_{\nu -1}, \tau^{(N)} \geq \nu\right]  + \sup_{S^{(1)}[1,N-1],\bm{S}^{(2)}[1,N-1]}\mathbb{E}_{\infty}\left[ \sum\limits_{i =\nu}^{N -1}  \Gamma_{\bm{S}} \left(i,\nu,d\right)  \mathbbm{1}_{\{i < \tau^{(N)} \}}  \bigg| \mathcal{F}_{\nu -1}, \tau^{(N)} \geq \nu\right] \nonumber 
 \\&\stackrel{(\text{f})}{=} \mathbb{E}_{\infty}\left[  \mathbbm{1}_{\{\tau^{(N)} \geq \nu\}}  \bigg| \mathcal{F}_{\nu -1}, \tau^{(N)} \geq \nu\right]  + \sup_{S^{(1)}[\nu,\nu+d-1],\bm{S}^{(2)}[\nu+d,N-1]}\mathbb{E}_{\infty}\left[ \sum\limits_{i =\nu}^{N -1}  \Gamma_{\bm{S}} \left(i,\nu,d\right)   \mathbbm{1}_{\{i < \tau^{(N)} \}}  \bigg| \mathcal{F}_{\nu -1}, \tau^{(N)} \geq \nu\right].
\end{align}
\normalsize
where (e) follows since the summation in the second expectation is from $i=\nu$ to $N-1$ which implies that only the first $N-1$ samples are involved in the calculation of $\Gamma_{\bm{S}} \left(i,\nu,d\right) $. Lastly, (f) follows because of the changepoints $\nu$ and $\nu+d$. Define
\begin{align}
B \triangleq \sup_{S^{(1)}[\nu,\nu+d-1],\bm{S}^{(2)}[\nu+d,N-1]}\mathbb{E}_{\infty}\left[ \sum\limits_{i =\nu}^{N -1}  \Gamma_{\bm{S}} \left(i,\nu,d\right)   \mathbbm{1}_{\{i < \tau^{(N)} \}}  \bigg| \mathcal{F}_{\nu -1}, \tau^{(N)} \geq \nu\right].
\end{align}
The proof continues by using the fact that the sup can be lower bound by the average. By using induction we can show that the sup can be completely removed. Consider the case of $d>0$. We first establish by induction that for $ 1 \leq \zeta \leq d$ we have that
\begin{align}
\label{eq:fact_1}
B &\nonumber\geq \mathbb{E}_\infty \left[\sum\limits_{i=\nu}^{\nu+\zeta-1} \mathcal{L}(i,\nu,\zeta)\mathbbm{1}_{\{\tau^{(N)} >i\}} \bigg| \mathcal{F}_{\nu-1},\tau^{(N)} \geq \nu\right] \\&+ \sup_{S^{(1)}[\nu +\zeta,\nu+d-1],\bm{S}^{(2)}[\nu+d,N-1]} \mathbb{E}_\infty \left[ \sum\limits_{i=\nu+\zeta}^{N-1} \mathcal{L}(\nu+\zeta-1,\nu,\zeta)\Gamma_{\bm{S}}(i,\nu+\zeta,d-\zeta) \mathbbm{1}_{\{\tau^{(N)} >i\}}\bigg| \mathcal{F}_{\nu-1},\tau^{(N)} \geq \nu \right].
\end{align}
Note that here by convention $\zeta = d$ implies that there is no sup over $S^{(1)}$. We first prove the claim for the case of $\zeta = 1$, establishing the basis of the induction. Note that for all $\ell \in [L]$ we have that
\begin{align}
B &\nonumber =  \sup_{S^{(1)}[\nu+1,\nu+d-1],\bm{S}^{(2)}[\nu+d,N-1]} \left\{ \sup_{S^{(1)}[\nu]} \mathbb{E}_{\infty}\left[ \sum\limits_{i =\nu}^{N -1}  \Gamma_{\bm{S}} \left(\nu,\nu,1\right) \Gamma_{\bm{S}} \left(i,\nu+1,d-1\right)  \mathbbm{1}_{\{i < \tau^{(N)} \}}  \bigg| \mathcal{F}_{\nu -1}, \tau^{(N)} \geq \nu\right]\right\} \nonumber  \\&=\sup_{S^{(1)}[\nu+1,\nu+d-1],\bm{S}^{(2)}[\nu+d,N-1]} \left\{ \sup_{S^{(1)}[\nu]} \mathbb{E}_{\infty}\left[ \sum\limits_{i =\nu}^{N -1}  \frac{f(X_{S^{(1)}[\nu]}[\nu])}{g(X_{S^{(1)}[\nu]}[\nu])}  \Gamma_{\bm{S}} \left(i,\nu+1,d-1\right)    \mathbbm{1}_{\{i < \tau^{(N)} \}}  \bigg| \mathcal{F}_{\nu -1}, \tau^{(N)} \geq \nu\right]\right\} \nonumber \\&\stackrel{(\text{g})}{\geq}\sup_{S^{(1)}[\nu+1,\nu+d-1],\bm{S}^{(2)}[\nu+d,N-1]} \left\{  \mathbb{E}_{\infty}\left[  \frac{f(X_{\ell}[\nu])}{g(X_{\ell}[\nu])} \sum\limits_{i =\nu}^{N -1}  \Gamma_{\bm{S}} \left(i,\nu+1,d-1\right)   \mathbbm{1}_{\{i < \tau^{(N)} \}}  \bigg| \mathcal{F}_{\nu -1}, \tau^{(N)} \geq \nu\right]\right\},
\end{align}
where (g) follows by the definition of the sup and since under $\mathbb{P}_\infty(\cdot)$,  $\Gamma_{\bm{S}} \left(i,\nu+1,d-1\right)$ is independent of $S^{(1)}[\nu]$ for $i \geq \nu$. By multiplying with $\frac{1}{L}$ and summing over $\ell \in [L]$ we have that
\begin{align}
\label{eq:eq_average_1}
B &\geq \sum\limits_{\ell=1}^L\sup_{S^{(1)}[\nu+1,\nu+d-1],\bm{S}^{(2)}[\nu+d,N-1]}  \nonumber \left\{  \mathbb{E}_{\infty}\left[ \frac{1}{L} \frac{f(X_{\ell}[\nu])}{g(X_{\ell}[\nu])} \sum\limits_{i =\nu}^{N -1}  \Gamma_{\bm{S}} \left(i,\nu+1,d-1\right)   \mathbbm{1}_{\{i < \tau^{(N)} \}}  \bigg| \mathcal{F}_{\nu -1}, \tau^{(N)} \geq \nu\right]\right\} \\& \geq\sup_{S^{(1)}[\nu+1,\nu+d-1],\bm{S}^{(2)}[\nu+d,N-1]} \sum\limits_{\ell=1}^L\left\{  \mathbb{E}_{\infty}\left[ \frac{1}{L} \frac{f(X_{\ell}[\nu])}{g(X_{\ell}[\nu])} \sum\limits_{i =\nu}^{N -1} \nonumber \Gamma_{\bm{S}} \left(i,\nu+1,d-1\right)   \mathbbm{1}_{\{i < \tau^{(N)} \}}  \bigg| \mathcal{F}_{\nu -1}, \tau^{(N)} \geq \nu\right]\right\} \\& \geq\sup_{S^{(1)}[\nu+1,\nu+d-1],\bm{S}^{(2)}[\nu+d,N-1]} \left\{  \mathbb{E}_{\infty}\left[ \left(\sum\limits_{\ell=1}^L \frac{1}{L} \frac{f(X_{\ell}[\nu])}{g(X_{\ell}[\nu])} \right)\sum\limits_{i =\nu}^{N -1} \nonumber \Gamma_{\bm{S}} \left(i,\nu+1,d-1\right)    \mathbbm{1}_{\{i < \tau^{(N)} \}}  \bigg| \mathcal{F}_{\nu -1}, \tau^{(N)} \geq \nu\right]\right\} \\& =\sup_{S^{(1)}[\nu+1,\nu+d-1],\bm{S}^{(2)}[\nu+d,N-1]} \left\{  \mathbb{E}_{\infty}\left[ \sum\limits_{i =\nu}^{N -1} \mathcal{L}(\nu,\nu,1)\nonumber \Gamma_{\bm{S}} \left(i,\nu+1,d-1\right)   \mathbbm{1}_{\{i < \tau^{(N)} \}}  \bigg| \mathcal{F}_{\nu -1}, \tau^{(N)} \geq \nu\right]\right\} \\& =\sup_{S^{(1)}[\nu+1,\nu+d-1],\bm{S}^{(2)}[\nu+d,N-1]} \left\{  \mathbb{E}_{\infty}\left[  \mathcal{L}(\nu,\nu,1)\nonumber \Gamma_{\bm{S}} \left(\nu,\nu+1,d-1\right)   \mathbbm{1}_{\{\nu < \tau^{(N)} \}}  \bigg| \mathcal{F}_{\nu -1}, \tau^{(N)} \geq \nu\right]\right\} \nonumber\\&\nonumber+ \sup_{S^{(1)}[\nu+1,\nu+d-1],\bm{S}^{(2)}[\nu+d,N-1]} \left\{  \mathbb{E}_{\infty}\left[ \sum\limits_{i =\nu+1}^{N -1} \mathcal{L}(\nu,\nu,1)\nonumber \Gamma_{\bm{S}} \left(i,\nu+1,d-1\right)   \mathbbm{1}_{\{i < \tau^{(N)} \}}  \bigg| \mathcal{F}_{\nu -1}, \tau^{(N)} \geq \nu\right]\right\} \\& =\sup_{S^{(1)}[\nu+1,\nu+d-1],\bm{S}^{(2)}[\nu+d,N-1]} \left\{  \mathbb{E}_{\infty}\left[  \mathcal{L}(\nu,\nu,1)\nonumber   \mathbbm{1}_{\{\nu < \tau^{(N)} \}}  \bigg| \mathcal{F}_{\nu -1}, \tau^{(N)} \geq \nu\right]\right\} \nonumber\\&\nonumber+ \sup_{S^{(1)}[\nu+1,\nu+d-1],\bm{S}^{(2)}[\nu+d,N-1]} \left\{  \mathbb{E}_{\infty}\left[ \sum\limits_{i =\nu+1}^{N -1} \mathcal{L}(\nu,\nu,1)\nonumber \Gamma_{\bm{S}} \left(i,\nu+1,d-1\right)  \mathbbm{1}_{\{i < \tau^{(N)} \}}  \bigg| \mathcal{F}_{\nu -1}, \tau^{(N)} \geq \nu\right]\right\}   \\& =   \mathbb{E}_{\infty}\left[  \mathcal{L}(\nu,\nu,1)\nonumber    \mathbbm{1}_{\{\nu < \tau^{(N)} \}}  \bigg| \mathcal{F}_{\nu -1}, \tau^{(N)} \geq \nu\right] \nonumber\\&+ \sup_{S^{(1)}[\nu+1,\nu+d-1],\bm{S}^{(2)}[\nu+d,N-1]} \left\{  \mathbb{E}_{\infty}\left[ \sum\limits_{i =\nu+1}^{N -1} \mathcal{L}(\nu,\nu,1)\Gamma_{\bm{S}} \left(i,\nu+1,d-1\right)    \mathbbm{1}_{\{i < \tau^{(N)} \}}  \bigg| \mathcal{F}_{\nu -1}, \tau^{(N)} \geq \nu\right]\right\},
\end{align}
which establishes the basis of the induction. Assume that the following claim holds for $3 \leq \zeta \leq d$:
\small
\begin{align}
\label{eq:eqa}
B &\nonumber\geq \mathbb{E}_\infty \left[\sum\limits_{i=\nu}^{\nu+\zeta-2} \mathcal{L}(i,\nu,\zeta-1)\mathbbm{1}_{\{\tau^{(N)} >i\}} \bigg| \mathcal{F}_{\nu-1},\tau^{(N)} \geq \nu\right] \\&+ \sup_{S^{(1)}[\nu+\zeta-1,\nu+d-1],\bm{S}^{(2)}[\nu+d,N-1]} \mathbb{E}_\infty \left[ \sum\limits_{i=\nu+\zeta-1}^{N-1} \mathcal{L}(\nu+\zeta-2,\nu,\zeta-1)\Gamma_{\bm{S}}(i,\nu+\zeta-1,d-\zeta+1) \mathbbm{1}_{\{\tau^{(N)} >i\}}\bigg| \mathcal{F}_{\nu-1},\tau^{(N)} \geq \nu \right].
\end{align}
\normalsize
Then, we have that
\small
\begin{align}
\label{eq:eqb}
& \nonumber\sup_{S^{(1)}[\nu+\zeta-1,\nu+d-1],\bm{S}^{(2)}[\nu+d,N-1]} \mathbb{E}_\infty \left[ \sum\limits_{i=\nu+\zeta-1}^{N-1} \mathcal{L}(\nu+\zeta-2,\nu,\zeta-1)\Gamma_{\bm{S}}(i,\nu+\zeta-1,d-\zeta+1)\mathbbm{1}_{\{\tau^{(N)} >i\}}\bigg| \mathcal{F}_{\nu-1},\tau^{(N)} \geq \nu \right] \nonumber \\\nonumber &= \sup_{S^{(1)}[\nu+\zeta,\nu+d-1],\bm{S}^{(2)}[\nu+d,N-1]} \\& \nonumber \left\{\sup_{S^{(1)}[\nu+\zeta-1]} \mathbb{E}_\infty \left[ \Gamma_{\bm{S}}(\nu+\zeta-1,\nu+\zeta-1,1) \sum\limits_{i=\nu+\zeta-1}^{N-1} \mathcal{L}(\nu+\zeta-2,\nu,\zeta-1)\Gamma_{\bm{S}}(i,\nu+\zeta,d-\zeta) \mathbbm{1}_{\{\tau^{(N)} >i\}}\bigg| \mathcal{F}_{\nu-1},\tau^{(N)} \geq \nu \right] \right\}\\\nonumber &= \sup_{S^{(1)}[\nu+\zeta,\nu+d-1],\bm{S}^{(2)}[\nu+d,N-1]} \\& \nonumber \left\{\sup_{S^{(1)}[\nu+\zeta-1]} \mathbb{E}_\infty \left[  \frac{f(X_{S^{(1)}[\nu+\zeta-1]}[\nu+\zeta-1])}{g(X_{S^{(1)}[\nu+\zeta-1]}[\nu+\zeta-1])} \sum\limits_{i=\nu+\zeta-1}^{N-1} \mathcal{L}(\nu+\zeta-2,\nu,\zeta-1)\Gamma_{\bm{S}}(i,\nu+\zeta,d-\zeta) \mathbbm{1}_{\{\tau^{(N)} >i\}}\bigg| \mathcal{F}_{\nu-1},\tau^{(N)} \geq \nu \right] \right\}\\& \nonumber \stackrel{(\text{h})}{\geq}  \sup_{S^{(1)}[\nu+\zeta,\nu+d-1],\bm{S}^{(2)}[\nu+d,N-1]} \\& \nonumber \mathbb{E}_\infty \left[ \mathcal{L}(\nu+\zeta-1,\nu+\zeta-1,1) \sum\limits_{i=\nu+\zeta-1}^{N-1} \mathcal{L}(\nu+\zeta-2,\nu,\zeta-1)\Gamma_{\bm{S}}(i,\nu+\zeta,d-\zeta) \mathbbm{1}_{\{\tau^{(N)} >i\}}\bigg| \mathcal{F}_{\nu-1},\tau^{(N)} \geq \nu \right] \\& \nonumber = \sup_{S^{(1)}[\nu+\zeta,\nu+d-1],\bm{S}^{(2)}[\nu+d,N-1]}  \mathbb{E}_\infty \left[  \sum\limits_{i=\nu+\zeta-1}^{N-1} \mathcal{L}(\nu+\zeta-1,\nu,\zeta)\Gamma_{\bm{S}}(i,\nu+\zeta,d-\zeta) \mathbbm{1}_{\{\tau^{(N)} >i\}}\bigg| \mathcal{F}_{\nu-1},\tau^{(N)} \geq \nu \right]  \\&= \sup_{S^{(1)}[\nu+\zeta,\nu+d-1],\bm{S}^{(2)}[\nu+d,N-1]}  \mathbb{E}_\infty \left[  \mathcal{L}(\nu+\zeta-1,\nu,\zeta)\Gamma_{\bm{S}}(\nu+\zeta-1,\nu+\zeta,d-\zeta) \mathbbm{1}_{\{\tau^{(N)} >\nu+\zeta-1\}}\bigg|  \nonumber \mathcal{F}_{\nu-1},\tau^{(N)} \geq \nu \right] \\&  \nonumber+ \sup_{S^{(1)}[\nu+\zeta,\nu+d-1],\bm{S}^{(2)}[\nu+d,N-1]}  \mathbb{E}_\infty \left[  \sum\limits_{i=\nu+\zeta}^{N-1} \mathcal{L}(\nu+\zeta-1,\nu,\zeta)\Gamma_{\bm{S}}(i,\nu+\zeta,d-\zeta) \mathbbm{1}_{\{\tau^{(N)} >i\}}\bigg| \mathcal{F}_{\nu-1},\tau^{(N)} \geq \nu \right] \\&=  \mathbb{E}_\infty \left[  \mathcal{L}(\nu+\zeta-1,\nu,\zeta) \mathbbm{1}_{\{\tau^{(N)} >\nu+\zeta-1\}}\bigg|  \nonumber \mathcal{F}_{\nu-1},\tau^{(N)} \geq \nu \right] \\&  + \sup_{S^{(1)}[\nu+\zeta,\nu+d-1],\bm{S}^{(2)}[\nu+d,N-1]}  \mathbb{E}_\infty \left[  \sum\limits_{i=\nu+\zeta}^{N-1} \mathcal{L}(\nu+\zeta-1,\nu,\zeta)\Gamma_{\bm{S}}(i,\nu+\zeta,d-\zeta) \mathbbm{1}_{\{\tau^{(N)} >i\}}\bigg| \mathcal{F}_{\nu-1},\tau^{(N)} \geq \nu \right] 
\end{align}
\normalsize
where (h) follows from bounding the sup by using the average as in \eqref{eq:eq_average_1}. Furthermore, since $\mathcal{L}(i,\nu,\zeta-1) = \mathcal{L}(i,\nu,\zeta)$ for $i\leq \nu+\zeta-2$ we have that
\begin{align}
\mathbb{E}_\infty \left[\sum\limits_{i=\nu}^{\nu+\zeta-2} \mathcal{L}(i,\nu,\zeta-1)\mathbbm{1}_{\{\tau^{(N)} >i\}} \bigg| \mathcal{F}_{\nu-1},\tau^{(N)} \geq \nu\right] = \mathbb{E}_\infty \left[\sum\limits_{i=\nu}^{\nu+\zeta-2} \mathcal{L}(i,\nu,\zeta)\mathbbm{1}_{\{\tau^{(N)} >i\}} \bigg| \mathcal{F}_{\nu-1},\tau^{(N)} \geq \nu\right]
\end{align}
which together with \eqref{eq:eqa} and \eqref{eq:eqb} establishes \eqref{eq:fact_1} by induction. 
Proceeding in a similar manner we can remove the sup by using the fact that the sup can be bounded by the average, as above. In particular, we establish by induction that for $1 \leq \zeta \leq N -\nu -d$ we have that
\begin{align}
\label{eq:fact_2}
B &\nonumber\geq \mathbb{E}_\infty \left[\sum\limits_{i=\nu}^{\nu+d+\zeta-1} \mathcal{L}(i,\nu,d)\mathbbm{1}_{\{\tau^{(N)} >i\}} \bigg| \mathcal{F}_{\nu-1},\tau^{(N)} \geq \nu\right] \\&+ \sup_{\bm{S}^{(2)}[\nu+d+\zeta,N-1]} \mathbb{E}_\infty \left[ \sum\limits_{i=\nu+d+\zeta}^{N-1} \mathcal{L}(\nu+d+\zeta-1,\nu,d)\Gamma_{\bm{S}}(i,\nu+d+\zeta,0) \mathbbm{1}_{\{\tau^{(N)} >i\}}\bigg| \mathcal{F}_{\nu-1},\tau^{(N)} \geq \nu \right].
\end{align}
To this end, note that by \eqref{eq:fact_1} for $\zeta=d$ we have that
\begin{align}
\label{eq:eqc}
B &\nonumber\geq \mathbb{E}_\infty \left[\sum\limits_{i=\nu}^{\nu+d-1} \mathcal{L}(i,\nu,d)\mathbbm{1}_{\{\tau^{(N)} >i\}} \bigg| \mathcal{F}_{\nu-1},\tau^{(N)} \geq \nu\right] \\&+ \sup_{\bm{S}^{(2)}[\nu+d,N-1]} \mathbb{E}_\infty \left[ \sum\limits_{i=\nu+d}^{N-1} \mathcal{L}(\nu+d-1,\nu,d)\Gamma_{\bm{S}}(i,\nu+d,0) \mathbbm{1}_{\{\tau^{(N)} >i\}}\bigg| \mathcal{F}_{\nu-1},\tau^{(N)} \geq \nu \right].
\end{align}
Analyzing the second term we have that for $(\ell,\ell') \in \text{comb}(L,2)$
\small
\begin{align}
&\sup_{\bm{S}^{(2)}[\nu+d,N-1]} \mathbb{E}_\infty \left[ \sum\limits_{i=\nu+d}^{N-1} \mathcal{L}(\nu+d-1,\nu,d)\Gamma_{\bm{S}}(i,\nu+d,0) \mathbbm{1}_{\{\tau^{(N)} >i\}}\bigg| \mathcal{F}_{\nu-1},\tau^{(N)} \geq \nu \right] \nonumber 
\\ \nonumber  
& =  \sup_{\bm{S}^{(2)}[\nu+d,N-1]} \mathbb{E}_\infty \left[ \sum\limits_{i=\nu+d}^{N-1} \mathcal{L}(\nu+d-1,\nu,d)\Gamma_{\bm{S}}(\nu+d,\nu+d,0)\Gamma_{\bm{S}}(i,\nu+d+1,0) \mathbbm{1}_{\{\tau^{(N)} >i\}}\bigg| \mathcal{F}_{\nu-1},\tau^{(N)} \geq \nu \right] \nonumber 
\\ \nonumber 
 & =  \sup_{\bm{S}^{(2)}[\nu+d+1,N-1]} \bigg\{ \sup_{\bm{S}^{(2)}[\nu+d]}  \mathbb{E}_\infty \bigg[ \frac{f(X_{S^{(2)}_1[\nu+d]}[\nu+d])f(X_{S^{(2)}_2[\nu+d]}[\nu+d])}{g(X_{S^{(2)}_1[\nu+d]}[\nu+d])g(X_{S^{(2)}_2[\nu+d]}[\nu+d])}\\& \sum\limits_{i=\nu+d}^{N-1} \mathcal{L}(\nu+d-1,\nu,d)\Gamma_{\bm{S}}(i,\nu+d+1,0) \mathbbm{1}_{\{\tau^{(N)} >i\}}\bigg| \mathcal{F}_{\nu-1},\tau^{(N)} \geq \nu \bigg]\bigg\} \nonumber \\& \geq \sup_{\bm{S}^{(2)}[\nu+d+1,N-1]} \bigg\{   \mathbb{E}_\infty \bigg[ \frac{f(X_{\ell}[\nu+d])f(X_{\ell'}[\nu+d])}{g(X_{\ell}[\nu+d])g(X_{\ell'}[\nu+d])}  \sum\limits_{i=\nu+d}^{N-1} \mathcal{L}(\nu+d-1,\nu,d)\Gamma_{\bm{S}}(i,\nu+d+1,0) \mathbbm{1}_{\{\tau^{(N)} >i\}}\bigg| \mathcal{F}_{\nu-1},\tau^{(N)} \geq \nu \bigg]\bigg\}.
\end{align}
\normalsize
By multiplying with $\frac{1}{\binom L2}$ and summing over $(\ell,\ell') \in \text{comb}(L,2)$ we have that 
\begin{align}
&\sup_{\bm{S}^{(2)}[\nu+d,N-1]} \mathbb{E}_\infty \left[ \sum\limits_{i=\nu+d}^{N-1} \mathcal{L}(\nu+d-1,\nu,d)\Gamma_{\bm{S}}(i,\nu+d,0) \mathbbm{1}_{\{\tau^{(N)} >i\}}\bigg| \mathcal{F}_{\nu-1},\tau^{(N)} \geq \nu \right] \nonumber 
\\&
 \geq \sum\limits_{(\ell, \ell') \,\in\, \text{comb}(L,2)}  \sup_{\bm{S}^{(2)}[\nu+d+1,N-1]} \bigg\{   \mathbb{E}_\infty \bigg[ \frac{1}{\binom L2}\frac{f(X_{\ell}[\nu+d])f(X_{\ell'}[\nu+d])}{g(X_{\ell}[\nu+d])g_{\ell'}(X_{\ell'}[\nu+d])} \nonumber \\& \sum\limits_{i=\nu+d}^{N-1} \mathcal{L}(\nu+d-1,\nu,d)\Gamma_{\bm{S}}(i,\nu+d+1,0) \mathbbm{1}_{\{\tau^{(N)} >i\}}\bigg| \mathcal{F}_{\nu-1},\tau^{(N)} \geq \nu \bigg]\bigg\}\nonumber  
\\& \geq   \sup_{\bm{S}^{(2)}[\nu+d+1,N-1]} \sum\limits_{(\ell, \ell') \,\in\, \text{comb}(L,2)} \bigg\{   \mathbb{E}_\infty \bigg[ \frac{1}{\binom L2}\frac{f(X_{\ell}[\nu+d])f(X_{\ell'}[\nu+d])}{g(X_{\ell}[\nu+d])g(X_{\ell'}[\nu+d])} \nonumber \\& \sum\limits_{i=\nu+d}^{N-1} \mathcal{L}(\nu+d-1,\nu,d)\Gamma_{\bm{S}}(i,\nu+d+1,0) \mathbbm{1}_{\{\tau^{(N)} >i\}}\bigg| \mathcal{F}_{\nu-1},\tau^{(N)} \geq \nu \bigg]\bigg\}\nonumber
\\& =  \sup_{\bm{S}^{(2)}[\nu+d+1,N-1]} \bigg\{   \mathbb{E}_\infty \bigg[ \sum\limits_{(\ell, \ell') \,\in \, \text{comb}(L,2)}  \left(\frac{1}{\binom L2}\frac{f(X_{\ell}[\nu+d])f(X_{\ell'}[\nu+d])}{g(X_{\ell}[\nu+d])g(X_{\ell'}[\nu+d])} \right) \nonumber \\& \sum\limits_{i=\nu+d}^{N-1} \mathcal{L}(\nu+d-1,\nu,d)\Gamma_{\bm{S}}(i,\nu+d+1,0) \mathbbm{1}_{\{\tau^{(N)} >i\}}\bigg| \mathcal{F}_{\nu-1},\tau^{(N)} \geq \nu \bigg]\bigg\}\nonumber
\\& =  \sup_{\bm{S}^{(2)}[\nu+d+1,N-1]} \bigg\{   \mathbb{E}_\infty \bigg[  \mathcal{L}\left(\nu+d,\nu+d,0\right)  \sum\limits_{i=\nu+d}^{N-1} \mathcal{L}(\nu+d-1,\nu,d)\Gamma_{\bm{S}}(i,\nu+d+1,0) \mathbbm{1}_{\{\tau^{(N)} >i\}}\bigg| \mathcal{F}_{\nu-1},\tau^{(N)} \geq \nu \bigg]\bigg\}\nonumber
\\& =   \sup_{\bm{S}^{(2)}[\nu+d+1,N-1]} \bigg\{   \mathbb{E}_\infty \bigg[  \sum\limits_{i=\nu+d}^{N-1} \mathcal{L}(\nu+d,\nu,d)\Gamma_{\bm{S}}(i,\nu+d+1,0)\mathbbm{1}_{\{\tau^{(N)} >i\}}\bigg| \mathcal{F}_{\nu-1},\tau^{(N)} \geq \nu \bigg]\bigg\}\nonumber
\\& =      \mathbb{E}_\infty \bigg[ \mathcal{L}(\nu+d,\nu,d)\mathbbm{1}_{\{\tau^{(N)} >\nu+d\}}\bigg| \mathcal{F}_{\nu-1},\tau^{(N)} \geq \nu \bigg]\nonumber
\\& +   \sup_{\bm{S}^{(2)}[\nu+d+1,N-1]} \bigg\{   \mathbb{E}_\infty \bigg[  \sum\limits_{i=\nu+d+1}^{N-1} \mathcal{L}(\nu+d,\nu,d)\Gamma_{\bm{S}}(i,\nu+d+1,0) \mathbbm{1}_{\{\tau^{(N)} >i\}}\bigg| \mathcal{F}_{\nu-1},\tau^{(N)} \geq \nu \bigg]\bigg\}
,
\end{align}
which together with \eqref{eq:eqc} implies the basis of the induction for \eqref{eq:fact_2}. Assume that the following equation holds for $ 3\leq\zeta\leq N - \nu -d$

\begin{align}
\label{eq:eqd}
B &\nonumber\geq \mathbb{E}_\infty \left[\sum\limits_{i=\nu}^{\nu+d+\zeta-2} \mathcal{L}(i,\nu,d)\mathbbm{1}_{\{\tau^{(N)} >i\}} \bigg| \mathcal{F}_{\nu-1},\tau^{(N)} \geq \nu\right] \\&+ \sup_{\bm{S}^{(2)}[\nu+d+\zeta-1,N-1]} \mathbb{E}_\infty \left[ \sum\limits_{i=\nu+d+\zeta-1}^{N-1} \mathcal{L}(\nu+d+\zeta-2,\nu,d)\Gamma_{\bm{S}}(i,\nu+d+\zeta-1,0) \mathbbm{1}_{\{\tau^{(N)} >i\}}\bigg| \mathcal{F}_{\nu-1},\tau^{(N)} \geq \nu \right].
\end{align}
By analyzing the second term we have that
\small
\begin{align}
& \nonumber \sup_{\bm{S}^{(2)}[\nu+d+\zeta-1,N-1]} \mathbb{E}_\infty \left[ \sum\limits_{i=\nu+d+\zeta-1}^{N-1} \mathcal{L}(\nu+d+\zeta-2,\nu,d)\Gamma_{\bm{S}}(i,\nu+d+\zeta-1,0) \mathbbm{1}_{\{\tau^{(N)} >i\}}\bigg| \mathcal{F}_{\nu-1},\tau^{(N)} \geq \nu \right]
\\& 
= \nonumber \sup_{\bm{S}^{(2)}[\nu+d+\zeta,N-1]}\bigg\{\sup_{\bm{S}^{(2)}[\nu+d+\zeta-1]}\mathbb{E}_\infty \bigg[
\\&
\nonumber \sum\limits_{i=\nu+d+\zeta-1}^{N-1} \mathcal{L}(\nu+d+\zeta-2,\nu,d)\Gamma_{\bm{S}}(\nu+d+\zeta-1,\nu+d+\zeta-1,0)\Gamma_{\bm{S}}(i,\nu+d+\zeta,0) \mathbbm{1}_{\{\tau^{(N)} >i\}}\bigg| \mathcal{F}_{\nu-1},\tau^{(N)} \geq \nu \bigg]\bigg\}
\\& 
\geq \nonumber \sup_{\bm{S}^{(2)}[\nu+d+\zeta,N-1]}\bigg\{\mathbb{E}_\infty \bigg[
\\&
\nonumber \sum\limits_{i=\nu+d+\zeta-1}^{N-1} \mathcal{L}(\nu+d+\zeta-2,\nu,d)\mathcal{L}(\nu+d+\zeta-1,\nu+d+\zeta-1,0)\Gamma_{\bm{S}}(i,\nu+d+\zeta,0) \mathbbm{1}_{\{\tau^{(N)} >i\}}\bigg| \mathcal{F}_{\nu-1},\tau^{(N)} \geq \nu \bigg]\bigg\}
\\& 
= \nonumber \sup_{\bm{S}^{(2)}[\nu+d+\zeta,N-1]}\bigg\{\mathbb{E}_\infty \bigg[
\nonumber \sum\limits_{i=\nu+d+\zeta-1}^{N-1} \mathcal{L}(\nu+d+\zeta-1,\nu,d)\Gamma_{\bm{S}}(i,\nu+d+\zeta,0) \mathbbm{1}_{\{\tau^{(N)} >i\}}\bigg| \mathcal{F}_{\nu-1},\tau^{(N)} \geq \nu \bigg]\bigg\}
\\& 
= \nonumber \mathbb{E}_\infty \bigg[
\nonumber  \mathcal{L}(\nu+d+\zeta-1,\nu,d)\mathbbm{1}_{\{\tau^{(N)} >\nu +d +\zeta-1\}}\bigg| \mathcal{F}_{\nu-1},\tau^{(N)} \geq \nu \bigg]
\\&+
  \sup_{\bm{S}^{(2)}[\nu+d+\zeta,N-1]}\bigg\{\mathbb{E}_\infty \bigg[
\sum\limits_{i=\nu+d+\zeta}^{N-1} \mathcal{L}(\nu+d+\zeta-1,\nu,d)\Gamma_{\bm{S}}(i,\nu+d+\zeta,0)\mathbbm{1}_{\{\tau^{(N)} >i\}}\bigg| \mathcal{F}_{\nu-1},\tau^{(N)} \geq \nu \bigg]\bigg\}
\end{align}
\normalsize
which together with \eqref{eq:eqd} implies \eqref{eq:fact_2} by induction. As a result by \eqref{eq:eqc} for $\zeta = N-\nu-d$ the following inequality follows:
\begin{align}
\label{eq:b_ineq}
B \geq \mathbb{E}_{\infty}\left[ \sum\limits_{i =\nu}^{N -1}  \mathcal{L}\left(i,\nu,d\right)  \mathbbm{1}_{\{i < \tau^{(N)} \}}  \bigg| \mathcal{F}_{\nu -1}, \tau^{(N)} \geq \nu\right],
\end{align}
when $d>0$. It can be easily shown that \eqref{eq:b_ineq} also holds when $d=0$. From \eqref{eq:b_ineq}, \eqref{eq:eqe} and \eqref{eq:eqf}, and by following similar steps as in \eqref{eq:eq_similar}, we then have that
\begin{align}
&\mathrm{WADD}_d(\tau) \geq  \mathbb{E}_{\infty}\left[  \mathbbm{1}_{\{\tau^{(N)} \geq \nu\}} \nonumber \bigg| \mathcal{F}_{\nu -1}, \tau^{(N)} \geq \nu\right]  +\mathbb{E}_{\infty}\left[ \sum\limits_{i =\nu}^{N -1}  \mathcal{L}\left(i,\nu,d\right)  \mathbbm{1}_{\{i < \tau^{(N)} \}}  \bigg| \mathcal{F}_{\nu -1}, \tau^{(N)} \geq \nu\right] 
 \\& = \mathbb{E}_{\infty}\left[  \mathbbm{1}_{\{\tau^{(N)} \geq \nu\}} \nonumber \bigg| \mathcal{F}_{\nu -1}, \tau^{(N)} \geq \nu\right]  + \mathbb{E}_{\infty}\left[ \sum\limits_{i =\nu+1}^{N }  \mathcal{L} \left(i-1,\nu,d\right) \mathbbm{1}_{\{i \leq \tau^{(N)} \}} \bigg| \mathcal{F}_{\nu -1}, \tau^{(N)} \geq \nu\right] 
 \\& 
 \nonumber  = \mathbb{E}_{\infty}\left[  \mathbbm{1}_{\{\tau^{(N)} \geq \nu\}} \mathcal{L}(\nu-1,\nu,d) \nonumber \bigg| \mathcal{F}_{\nu -1}, \tau^{(N)} \geq \nu\right]  + \mathbb{E}_{\infty}\left[ \sum\limits_{i =\nu+1}^{N }  \mathcal{L} \left(i-1,\nu,d\right) \mathbbm{1}_{\{i \leq \tau^{(N)} \}} \bigg| \mathcal{F}_{\nu -1}, \tau^{(N)} \geq \nu\right] 
   \\& 
 =  \mathbb{E}_{\infty}\left[ \sum\limits_{i =\nu}^{N }  \mathcal{L} \left(i-1,\nu,d\right) \mathbbm{1}_{\{i \leq \tau^{(N)} \}} \bigg| \mathcal{F}_{\nu -1}, \tau^{(N)} \geq \nu\right]  =   \sum\limits_{i =\nu}^{N } \mathbb{E}_{\infty}\left[ \mathcal{L} \left(i-1,\nu,d\right) \mathbbm{1}_{\{i \leq \tau^{(N)} \}} \bigg| \mathcal{F}_{\nu -1}, \tau^{(N)} \geq \nu\right]   \nonumber
   \\&  
    =\sum\limits_{i =\nu}^{N } \overline{\mathbb{E}}_{\nu,d}\left[  \mathbbm{1}_{\{i \leq \tau^{(N)} \}} \bigg| \mathcal{F}_{\nu -1}, \tau^{(N)} \geq \nu\right]    =  \overline{\mathbb{E}}_{\nu,d}\left[ \sum\limits_{i =\nu}^{N } \mathbbm{1}_{\{i \leq \tau^{(N)} \}} \bigg| \mathcal{F}_{\nu -1}, \tau^{(N)} \geq \nu\right] \nonumber    
    \\& = \overline{\mathbb{E}}_{\nu,d}\left[ \sum\limits_{i =\nu}^{\infty } \mathbbm{1}_{\{i \leq \tau^{(N)} \}} \bigg| \mathcal{F}_{\nu -1}, \tau^{(N)} \geq \nu\right] =  \overline{\mathbb{E}}_{\nu,d}[\tau^{(N)} - \nu +1 | \mathcal{F}_{\nu -1}, \tau^{(N)} \geq \nu].
\end{align}
 From the Monotone Convergence Theorem, since $\tau^{(N)}-\nu+1$ and $\mathbbm{1}_{\{\tau^{(N)}\geq \nu\}}$ are non-decreasing with $N$, we have that
\begin{align}
\label{eq:eq28}
 &\lim_{N \rightarrow \infty} \overline{\mathbb{E}}_{\nu,d}[\tau^{(N)} - \nu +1 | \mathcal{F}_{\nu -1}, \tau^{(N)} \geq \nu] = \lim_{N \rightarrow \infty}\frac{\overline{\mathbb{E}}_{\nu,d} \left[(\tau^{(N)} - \nu +1)\mathbbm{1}_{\{\tau^{(N)}\geq \nu\}}  | \mathcal{F}_{\nu -1} \right]}{\overline{\mathbb{E}}_{\nu,d} \left[\mathbbm{1}_{\{\tau^{(N)}\geq \nu\}} | \mathcal{F}_{\nu -1} \right] }\nonumber\\&= \frac{\overline{\mathbb{E}}_{\nu,d} \left[\lim_{N \rightarrow \infty}(\tau^{(N)} - \nu +1)\mathbbm{1}_{\{\tau^{(N)}\geq \nu\}}  | \mathcal{F}_{\nu -1} \right]}{\overline{\mathbb{E}}_{\nu,d} \left[\lim_{N \rightarrow \infty}\mathbbm{1}_{\{\tau^{(N)}\geq \nu\}} | \mathcal{F}_{\nu -1} \right] }  =\frac{\overline{\mathbb{E}}_{\nu,d} \left[(\tau - \nu +1 )\mathbbm{1}_{\{\tau\geq \nu\}} | \mathcal{F}_{\nu -1} \right]}{\overline{\mathbb{E}}_{\nu,d} \left[\mathbbm{1}_{\{\tau\geq \nu\}} | \mathcal{F}_{\nu -1}\right] }\nonumber \\&=\overline{\mathbb{E}}_{\nu,d} \left[\tau - \nu +1 |\mathcal{F}_{\nu -1}, \tau^{(N)} \geq \nu \right].
\end{align}
As a result, by taking the $\sup$ over $\nu$ and the ess $\sup$ we have that for any stopping time $\tau$ and for $d\geq 0$
\begin{align}
\label{eq:WADD_ineq}
\mathrm{WADD}_d(\tau) \geq \overline{\mathrm{WADD}}_d(\tau).
\end{align}
\end{proof}

From Lemma 1 and the universal lower bound in \cite{Zou2017} we then have the following theorem:
\begin{theorem}
Consider the QCD problem described by \eqref{eq:stat_model} and \eqref{eq:WADD_growing}. Assume that for $c \geq 0$
\begin{align}
\label{eq:d_scaling}
d \sim c \frac{\log \gamma}{I_{1}}
\end{align}
as $\gamma \rightarrow \infty$.
Then, we have that

\begin{align}
\inf_{\tau \in C_\gamma} \mathrm{WADD}_d(\tau)  \geq \left\{
\begin{array}{ll}
      \frac{\log\gamma}{I_{1}}(1-o(1)) & c\geq 1\\
      \log \gamma \left(\frac{1-c}{I_{2}}+\frac{c}{I_{1}} \right)(1-o(1))& c <1 \\
\end{array} 
\right. 
\end{align}
as $\gamma \rightarrow \infty$.
\begin{proof}
From \eqref{eq:WADD_ineq} we have that for any $d\geq 0$
\begin{align}
\label{eq:eql}
\inf_{\tau \in C_\gamma} \mathrm{WADD}_d(\tau) \geq \inf_{\tau \in C_\gamma} \overline{\mathrm{WADD}}_d(\tau).
\end{align}
The proof then follows directly by using \eqref{eq:eql} and the lower bound for the instance of the transient QCD problem \cite{Zou2017} described in \eqref{eq:stat_model}.
\end{proof}
\end{theorem}

To establish the asymptotic optimality of the proposed mixture-WD-CuSum test we have to show that equality in \eqref{eq:lemma_1_eq} is attained when $\tau = \tau_W$. We continue with the proof of Lemma 2.
\begin{proof}[Proof of Lemma 2]
 Note that due to the symmetry of the test, the delay of $\tau_W$ is independent of $\bm{S}$. Furthermore, because of the structure of the test the worst case delay is achieved at $\nu=1$. As a result, we have that for any $\bm{S}$, $d$
\begin{align}
\mathrm{WADD}_d(\tau_W) & =  \sup_{\bm {S}} \sup_{\nu \geq 1} \mathrm{ess}\hspace{-0.6pt} \sup \mathbb{E}_{\nu,d}^{\bm{S}}[\tau_W - \nu +1 | \mathcal{F}_{\nu -1}, \tau_W \geq \nu] = \sup_{\nu \geq 1} \mathrm{ess}\hspace{-0.6pt} \sup \mathbb{E}_{\nu,d}^{\bm{S}}[\tau_W - \nu +1 | \mathcal{F}_{\nu -1}, \tau_W \geq \nonumber \nu] \\&= \mathbb{E}_{1,d}^{\bm{S}}[\tau_W ]  = \lim\limits_{N \rightarrow \infty} \sum\limits_{i=1}^N i \mathbb{P}_{1,d}^{\bm{S}}( \tau_W = i).
\end{align}
Analyzing the summation in the limit for $d>0$ and $N>d$ we have
\begin{align}
\label{eq:eqk}
\sum\limits_{i=1}^N i \mathbb{P}_{1,d}^{\bm{S}}\left( \tau_W = i\right) = \sum\limits_{i=1}^N i \mathbb{E}_{1,d}^{\bm{S}}\left[\mathbbm{1}_{\{\tau_W = i\}}\right]= \sum\limits_{i=1}^N i \mathbb{E}_{\infty}\left[\Gamma_{\bm{S}}(i,1,d)\mathbbm{1}_{\{\tau_W = i\}} \right].
\end{align}
Following, we show that for $1 \leq \zeta \leq d$ we have that
\begin{align}
\label{eq:eqg}
\sum\limits_{i=1}^N i \mathbb{E}_{\infty}\left[\Gamma_{\bm{S}}(i,1,d)\mathbbm{1}_{\{\tau_W = i\}} \right] = \sum\limits_{i=1}^\zeta i \mathbb{E}_\infty \left[ \mathcal{L}\left(i,1,\zeta\right) \mathbbm{1}_{\{\tau_W=i\}} \right] + \sum_{i= \zeta +1}^N i \mathbb{E}_\infty \left[ \mathcal{L}\left(\zeta,1,\zeta\right) \Gamma_{\bm{S}}\left(i,\zeta+1,d-\zeta \right) \mathbbm{1}_{\{\tau_W=i\}}\right].
\end{align}
We initially prove the claim for the case of $\zeta=1$. In particular, we have that
\begin{align}
\sum\limits_{i=1}^N i \mathbb{E}_{\infty}\left[\Gamma_{\bm{S}}(i,1,d)\mathbbm{1}_{\{\tau_W = i\}} \right]& = \sum\limits_{i=1}^N i \mathbb{E}_{\infty}\left[\Gamma_{ \bm{S}}(1,1,1)\Gamma_{\bm{S}}(i,2,d-1)\mathbbm{1}_{\{\tau_W = i\}} \right] \nonumber \\& = \sum\limits_{i=1}^N i \mathbb{E}_{\infty}\left[\log \frac{f(X_{S^{(1)}[1]}[1])}{g(X_{S^{(1)}[1]}[1])} \Gamma_{\bm{S}}(i,2,d-1)\mathbbm{1}_{\{\tau_W = i\}} \right].
\end{align}
Note that for all $\ell \in [L]$
\begin{align}
\frac{f(X_{S^{(1)}[1]}[1])}{g(X_{S^{(1)}[1]}[1])} \stackrel{d}{=} \frac{f(X_{\ell}[1])}{g(X_{\ell}[1])}
\end{align}
under $f_\infty(\cdot)$, which implies that for all $\ell \in [L]$
\begin{align}
\sum\limits_{i=1}^N i \mathbb{E}_{\infty}\left[\log \frac{f(X_{S^{(1)}[1]}[1])}{g(X_{S^{(1)}[1]}[1])} \Gamma_{ \bm{S}}(i,2,d-1)\mathbbm{1}_{\{\tau_W = i\}} \right] = \sum\limits_{i=1}^N i \mathbb{E}_{\infty}\left[\log \frac{f(X_{\ell}[1])}{g(X_{\ell}[1])} \Gamma_{ \bm{S}}(i,2,d-1)\mathbbm{1}_{\{\tau_W = i\}} \right]
\end{align}
which in turn, by averaging over the location of the anomalies, implies that
\small
\begin{align}
&\sum\limits_{i=1}^N i \mathbb{E}_{\infty}\left[\Gamma_{\bm{S}}(i,1,d)\mathbbm{1}_{\{\tau_W = i\}} \right]  = \sum\limits_{\ell=1}^L \frac{1}{L} \left( \sum\limits_{i=1}^N i \mathbb{E}_{\infty}\left[\Gamma_{\bm{S}}(i,1,d)\mathbbm{1}_{\{\tau_W = i\}} \right]\right)  \nonumber \\& = \sum\limits_{\ell=1}^L \frac{1}{L} \left(\sum\limits_{i=1}^N i \mathbb{E}_{\infty}\left[\log \frac{f(X_{\ell}[1])}{g(X_{\ell}[1])} \Gamma_{ \bm{S}}(i,2,d-1)\mathbbm{1}_{\{\tau_W = i\}} \right]\right) \nonumber \\& = \sum\limits_{i=1}^N i \mathbb{E}_{\infty}\left[ \sum\limits_{\ell=1}^L \left( \frac{1}{L}\log \frac{f(X_{\ell}[1])}{g(X_{\ell}[1])}\right)  \Gamma_{\bm{S}}(i,2,d-1)\mathbbm{1}_{\{\tau_W = i\}} \right] \nonumber = \sum\limits_{i=1}^N i \mathbb{E}_{\infty}\left[\mathcal{L}(1,1,1) \Gamma_{\bm{S}}(i,2,d-1)\mathbbm{1}_{\{\tau_W = i\}} \right]  \\&=  \mathbb{E}_{\infty}\left[\mathcal{L}(1,1,1)\mathbbm{1}_{\{\tau_W = 1\}} \right]   + \sum\limits_{i=2}^N i \mathbb{E}_{\infty}\left[\mathcal{L}(1,1,1) \Gamma_{\bm{S}}(i,2,d-1)\mathbbm{1}_{\{\tau_W = i\}} \right]  
\end{align}
\normalsize
establishing the basis of the induction. Assume that the following equation holds for $3 \leq \zeta \leq d$
\begin{align}
\sum\limits_{i=1}^N i \mathbb{E}_{\infty}\left[\Gamma_{ \bm{S}}(i,1,d)\mathbbm{1}_{\{\tau_W = i\}} \right] &= \sum\limits_{i=1}^{\zeta-1} i \mathbb{E}_\infty \left[ \mathcal{L}\left(i,1,\zeta-1\right) \mathbbm{1}_{\{\tau_W=i\}} \right] \nonumber\\&+ \sum_{i= \zeta }^N i \mathbb{E}_\infty \left[ \mathcal{L}\left(\zeta-1,1,\zeta-1\right) \Gamma_{\bm{S}}\left(i,\zeta,d-\zeta+1 \right) \mathbbm{1}_{\{\tau_W=i\}}\right].
\end{align}
We then have that
\small
\begin{align}
\sum_{i= \zeta }^N i \mathbb{E}_\infty \left[ \mathcal{L}\left(\zeta-1,1,\zeta-1\right) \Gamma_{ \bm{S}}\left(i,\zeta,d-\zeta+1 \right) \mathbbm{1}_{\{\tau_W=i\}}\right] = \sum_{i= \zeta }^N i \mathbb{E}_\infty \left[ \mathcal{L}\left(\zeta-1,1,\zeta-1\right)\Gamma_{\bm{S}}\left(\zeta,\zeta,1 \right) \Gamma_{\bm{S}}\left(i,\zeta+1,d-\zeta\right) \mathbbm{1}_{\{\tau_W=i\}}\right]
\end{align}
\normalsize
which implies that
\begin{align}
&\sum\limits_{i=1}^N i \mathbb{E}_{\infty}\left[\Gamma_{ \bm{S}}(i,1,d)\mathbbm{1}_{\{\tau_W = i\}} \right] = \sum\limits_{\ell=1}^L \frac{1}{L}\left(\sum\limits_{i=1}^N i \mathbb{E}_{\infty}\left[\Gamma_{\bm{S}}(i,1,d)\mathbbm{1}_{\{\tau_W = i\}} \right] \right)= \nonumber\\&=\sum\limits_{i=1}^{\zeta-1} i \mathbb{E}_\infty \left[ \mathcal{L}\left(i,1,\zeta-1\right) \mathbbm{1}_{\{\tau_W=i\}} \right] + \sum_{i= \zeta }^N i \mathbb{E}_\infty \left[ \mathcal{L}\left(\zeta-1,1,\zeta-1\right)\mathcal{L}\left(\zeta,\zeta,1 \right) \Gamma_{\bm{S}}\left(i,\zeta+1,d-\zeta\right) \mathbbm{1}_{\{\tau_W=i\}}\right]\nonumber \\&= \sum\limits_{i=1}^{\zeta-1} i \mathbb{E}_\infty \left[ \mathcal{L}\left(i,1,\zeta\right) \mathbbm{1}_{\{\tau_W=i\}} \right] + \sum_{i= \zeta }^N i \mathbb{E}_\infty \left[ \mathcal{L}\left(\zeta,1,\zeta\right)\Gamma_{\bm{S}}\left(i,\zeta+1,d-\zeta\right) \mathbbm{1}_{\{\tau_W=i\}}\right] \nonumber \\&= \sum\limits_{i=1}^{\zeta} i \mathbb{E}_\infty \left[ \mathcal{L}\left(i,1,\zeta\right) \mathbbm{1}_{\{\tau_W=i\}} \right] + \sum_{i= \zeta+1 }^N i \mathbb{E}_\infty \left[ \mathcal{L}\left(\zeta,1,\zeta\right)\Gamma_{ \bm{S}}\left(i,\zeta+1,d-\zeta\right) \mathbbm{1}_{\{\tau_W=i\}}\right]
\end{align}
which proves \eqref{eq:eqg} by induction. Proceeding in a similar way by averaging over the double anomaly we can easily establish that for $2 \leq \zeta \leq N-d+1$ we have that
\begin{align}
\label{eq:eqh}
\sum\limits_{i=1}^N i \mathbb{E}_{\infty}\left[\Gamma_{ \bm{S}}(i,1,d)\mathbbm{1}_{\{\tau_W = i\}} \right] &= \sum\limits_{i=1}^{d+\zeta-1} i \mathbb{E}_\infty \left[ \mathcal{L}\left(i,1,d\right) \mathbbm{1}_{\{\tau_W=i\}} \right] \nonumber\\&+ \sum_{i= d+\zeta}^N i \mathbb{E}_\infty \left[ \mathcal{L}\left(d+\zeta-1,1,d\right)\Gamma_{ \bm{S}}\left(i,d+\zeta,0\right) \mathbbm{1}_{\{\tau_W=i\}}\right].
\end{align}
In particular, from eq. \eqref{eq:eqg} for $\zeta = d$ we have that
\begin{align}
\label{eq:eqi}
\sum\limits_{i=1}^N i \mathbb{E}_{\infty}\left[\Gamma_{\bm{S}}(i,1,d)\mathbbm{1}_{\{\tau_W = i\}} \right] &= \sum\limits_{i=1}^d i \mathbb{E}_\infty \left[ \mathcal{L}\left(i,1,d\right) \mathbbm{1}_{\{\tau_W=i\}} \right] \nonumber \\&+ \sum_{i= d +1}^N i \mathbb{E}_\infty \left[ \mathcal{L}\left(d,1,d\right) \Gamma_{ \bm{S}}\left(i,d+1,0\right) \mathbbm{1}_{\{\tau_W=i\}}\right].
\end{align}
We then have for the second term that for all $(\ell,\ell') \in \text{comb}(L,2)$
\begin{align}
&\sum_{i= d +1}^N i \mathbb{E}_\infty \left[ \mathcal{L}\left(d,1,d\right) \Gamma_{ \bm{S}}\left(i,d+1,0\right) \mathbbm{1}_{\{\tau_W=i\}}\right] = \sum_{i= d +1}^N i \mathbb{E}_\infty \left[ \mathcal{L}\left(d,1,d\right) \Gamma_{ \bm{S}}\left(d+1,d+1,0\right)\Gamma_{\bm{S}}\left(i,d+2,0\right) \mathbbm{1}_{\{\tau_W=i\}}\right] \nonumber
\\&
= \sum_{i= d +1}^N i \mathbb{E}_\infty \left[ \mathcal{L}\left(d,1,d\right)  \frac{f(X_{\ell}[d+1])f(X_{\ell'}[d+1])}{g(X_{\ell}[d+1])g(X_{\ell'}[d+1])} \Gamma_{\bm{S}}\left(i,d+2,0\right) \mathbbm{1}_{\{\tau_W=i\}}\right].
\end{align}
By averaging over the double anomaly we have that
\begin{align}
&\nonumber\sum_{i= d +1}^N i \mathbb{E}_\infty \left[ \mathcal{L}\left(d,1,d\right) \Gamma_{ \bm{S}}\left(i,d+1,0\right) \mathbbm{1}_{\{\tau_W=i\}}\right] \\&\nonumber = \sum\limits_{(\ell, \ell') \,\in \, \text{comb}(L,2)}\sum_{i= d +1}^N i \mathbb{E}_\infty \left[ \mathcal{L}\left(d,1,d\right)  \frac{1}{\binom L2}\frac{f(X_{\ell}[d+1])f(X_{\ell'}[d+1])}{g(X_{\ell}[d+1])g(X_{\ell'}[d+1])} \Gamma_{ \bm{S}}\left(i,d+2,0\right) \mathbbm{1}_{\{\tau_W=i\}}\right]
\\&\nonumber = \sum_{i= d +1}^N i \mathbb{E}_\infty \left[ \mathcal{L}\left(d,1,d\right) \left(\sum\limits_{(\ell, \ell') \in \text{comb}(L,2)}\frac{1}{\binom L2}\frac{f(X_{\ell}[d+1])f(X_{\ell'}[d+1])}{g(X_{\ell}[d+1])g(X_{\ell'}[d+1])} \right)\Gamma_{ \bm{S}}\left(i,d,0\right) \mathbbm{1}_{\{\tau_W=i\}}\right]
\\&\nonumber = \sum_{i= d +1}^N i \mathbb{E}_\infty \left[ \mathcal{L}\left(d,1,d\right) \mathcal{L}\left(d+1,d+1,0\right)\Gamma_{ \bm{S}}\left(i,d+2,0\right)  \mathbbm{1}_{\{\tau_W=i\}}\right]
\\&\nonumber = \sum_{i= d +1}^N i \mathbb{E}_\infty \left[ \mathcal{L}\left(d+1,1,d\right) \Gamma_{ \bm{S}}\left(i,d+2,0\right)  \mathbbm{1}_{\{\tau_W=i\}}\right]
\\&\nonumber = (d+1) \mathbb{E}_\infty \left[ \mathcal{L}\left(d+1,1,d\right)   \mathbbm{1}_{\{\tau_W=d+1\}}\right]+\sum_{i= d +2}^N i \mathbb{E}_\infty \left[ \mathcal{L}\left(d+1,1,d\right) \Gamma_{\bm{S}}\left(i,d+2,0\right)  \mathbbm{1}_{\{\tau_W=i\}}\right].
\end{align}
which together with \eqref{eq:eqi} establishes \eqref{eq:eqh} for $\zeta=2$. Assume that the following equation holds for $4 \leq \zeta \leq N-d+1$
\begin{align}
\label{eq:eqj}
\sum\limits_{i=1}^N i \mathbb{E}_{\infty}\left[\Gamma_{\bm{S}}(i,1,d)\mathbbm{1}_{\{\tau_W = i\}} \right] &= \sum\limits_{i=1}^{d+\zeta-2} i \mathbb{E}_\infty \left[ \mathcal{L}\left(i,1,d\right) \mathbbm{1}_{\{\tau_W=i\}} \right]\nonumber \\& + \sum_{i= d+\zeta-1}^N i \mathbb{E}_\infty \left[ \mathcal{L}\left(d+\zeta-2,1,d\right)\Gamma_{\bm{S}}\left(i,d+\zeta-1,0\right) \mathbbm{1}_{\{\tau_W=i\}}\right].
\end{align}
Analyzing the second term we have
\begin{align}
 &\sum_{i= d+\zeta-1}^N i \mathbb{E}_\infty \left[ \mathcal{L}\left(d+\zeta-2,1,d\right)\Gamma_{\bm{S}}\left(i,d+\zeta-1,0\right) \mathbbm{1}_{\{\tau_W=i\}}\right] \nonumber
 \\&=
 \nonumber 
  \sum_{i= d+\zeta-1}^N i \mathbb{E}_\infty \left[ \mathcal{L}\left(d+\zeta-2,1,d\right)\Gamma_{\bm{S}}\left(d+\zeta-1,d+\zeta-1,0\right)\Gamma_{\bm{S}}\left(i,d+\zeta,0\right) \mathbbm{1}_{\{\tau_W=i\}}\right] 
   \\&=
 \nonumber 
  \sum_{i= d+\zeta-1}^N i \mathbb{E}_\infty \left[ \mathcal{L}\left(d+\zeta-2,1,d\right)\mathcal{L}\left(d+\zeta-1,d+\zeta-1,0\right) \Gamma_{\bm{S}}\left(i,d+\zeta,0\right) \mathbbm{1}_{\{\tau_W=i\}}\right] 
  \\&=
  \sum_{i= d+\zeta-1}^N i \mathbb{E}_\infty \left[ \mathcal{L}\left(d+\zeta-1,1,d\right) \Gamma_{\bm{S}}\left(i,d+\zeta,0\right) \mathbbm{1}_{\{\tau_W=i\}}\right] \nonumber
   \\&=
   (d+\zeta-1) \mathbb{E}_\infty \left[ \mathcal{L}\left(d+\zeta-1,1,d\right)  \mathbbm{1}_{\{\tau_W=d+\zeta-1\}}\right] + 
  \sum_{i= d+\zeta}^N i \mathbb{E}_\infty \left[ \mathcal{L}\left(d+\zeta-1,1,d\right) \Gamma_{\bm{S}}\left(i,d+\zeta,0\right) \mathbbm{1}_{\{\tau_W=i\}}\right] 
 ,
\end{align}
which together with \eqref{eq:eqj} implies \eqref{eq:eqh}. For $\zeta = N-d+1$ it is easy to see that \eqref{eq:eqh} then implies
\begin{align}
\label{eq:eqm}
\sum\limits_{i=1}^N i \mathbb{E}_{\infty}\left[\Gamma_{\bm{S}}(i,1,d)\mathbbm{1}_{\{\tau_W = i\}} \right] =\sum\limits_{i=1}^N i \mathbb{E}_{\infty}\left[\mathcal{L}(i,1,d)\mathbbm{1}_{\{\tau_W = i\}} \right]
\end{align}
for $d>0$. Following the same process we can prove the same equality when $d=0$. As a result, we have by \eqref{eq:eqk} and \eqref{eq:eqm} that for $d\geq 0$
\begin{align}
\sum\limits_{i=1}^N i \mathbb{P}_{1,d}^{\bm{S}}\left( \tau_W = i\right) =  \sum\limits_{i=1}^N i \mathbb{E}_{\infty}\left[\mathcal{L}(i,1,d)\mathbbm{1}_{\{\tau_W = i\}} \right] = \sum\limits_{i=1}^N i \overline{\mathbb{E}}_{1,d}\left[\mathbbm{1}_{\{\tau_W = i\}}\right] =\sum\limits_{i=1}^N i \overline{\mathbb{P}}_{1,d}\left(\tau_W = i\right),
\end{align}
which in turn, by taking the limit as $N \rightarrow \infty$ implies that for all $\bm{S}$, $\bm{d}$ we have that
\begin{align}
\mathrm{WADD}_d(\tau_W) = \mathbb{E}^{\bm{S}}_{1,d}[\tau_W] =  \overline{\mathbb{E}}_{1,d}[\tau_W] \stackrel{(\text{i})}{=} \overline{\mathrm{WADD}}_d(\tau_W),
\end{align}
where (i) follows since the WD-CuSum attains the $\overline{\mathrm{WADD}}_d$ at $\nu=1$.
\end{proof}

Using Lemma 2 and the upper bound on the delay of the WD-CuSum procedure \cite{Zou2017} we have the following theorem:
\begin{theorem}
Consider the stopping time specified by \eqref{eq:stat_WD} - \eqref{eq:stop_WD}. Assume that there exists a constant $c' \geq 0$ such that
\begin{align}
d\sim c' \frac{b}{I_{1}}
\end{align}
as $b \rightarrow \infty$.
Furthermore, assume that
\begin{align}
\rho_1 \rightarrow 0 \text{    and    } \frac{\log \rho_1 }{b} \rightarrow 0
\end{align}
as $b \rightarrow \infty$. We then have that for all $d\geq 0$
\begin{align}
 \mathrm{WADD}_d(\tau_W)  \leq \left\{
\begin{array}{ll}
      \frac{b}{I_1}(1+o(1)) & c'\geq 1\\
      b \left(\frac{1-c'}{I_2}+\frac{c'}{I_1} \right)(1+o(1))& c' <1 \\
\end{array} 
\right. 
\end{align}
as $b \rightarrow \infty$.
\begin{proof}
The result follows directly from Lemma 3 and the asymptotic upper bound on the delay of the WD-CuSum \cite{Zou2017}.
\end{proof}
\end{theorem}

Finally, we establish the asymptotic optimality of our proposed test:

\begin{proof}[Proof of Theorem 1] 
i) follows directly from the MTFA lower bound on the WD-CuSum test \cite{Zou2017}. ii) follows from Theorem 2 and 3 and since $c = c'$ when $b=\log \gamma$.
\end{proof}

\end{document}